\documentclass[10pt]{article}
\usepackage{graphicx}
\usepackage{amssymb}
\usepackage{amsthm}
\usepackage{epstopdf}
\usepackage{color}
\usepackage{xcolor}
\usepackage{pstricks}
\usepackage{mdframed}
\usepackage{epic,eepic,epsfig,amssymb,amsmath,amsthm,graphics}
\usepackage{authblk}
\usepackage{tikz}
\usepackage{tikz-3dplot}

\usepackage{amsmath}
\usepackage{hyperref}
\usepackage{amsfonts}

\def\phi{{\varphi}}

\DeclareSymbolFont{AMSb}{U}{msb}{m}{n}
\DeclareMathSymbol{\N}{\mathbin}{AMSb}{"4E}
\DeclareMathSymbol{\Z}{\mathbin}{AMSb}{"5A}
\DeclareMathSymbol{\R}{\mathbin}{AMSb}{"52}
\DeclareMathSymbol{\Q}{\mathbin}{AMSb}{"51}
\DeclareMathSymbol{\I}{\mathbin}{AMSb}{"49}
\DeclareMathSymbol{\C}{\mathbin}{AMSb}{"43}

\DeclareMathOperator*{\argmax}{argmax}
\def\be{\begin{equation}}
\def\ee{\end{equation}}
\def\ber{\begin{eqnarray}}
\def\eer{\end{eqnarray}}

\def\beq{\begin{equation}}
\def\eeq{\end{equation}}

 \newtheorem{proposition}{Proposition}[section]
\newtheorem{theorem}{Theorem}[section]
\newtheorem{lemma}[theorem]{Lemma}

\newtheorem{remark}[theorem]{Remark}
\newtheorem{example}[theorem]{Example}

\newtheorem{definition}[theorem]{Definition}

\newtheorem{corollary}[theorem]{Corollary}

\title{Uniqueness of optimal plans for  multi-marginal    mass transport  problems via a reduction argument
\footnote{A.M. is pleased to acknowledge the support of the  National Sciences and Engineering Research Council of Canada.}}

\author{  Mohammad Ali Ahmadpoor \thanks{School of Mathematics and Statistics,  Carleton University, Ottawa, ON,  Canada, mohammadaliahmadpoo@cmail.carleton.ca} \quad  and \quad Abbas Moameni\thanks{School of Mathematics and Statistics,  Carleton University, Ottawa, ON,  Canada, momeni@math.carleton.ca}
}
\date{}

\begin{document}

\maketitle
\begin{abstract}
In this paper,  we investigate the  uniqueness of  an optimal mass transport problem with $N$-marginals for  $N\geq 3$  by transforming it into a lower marginal optimal transportation problem.  Namely,   for  a family of probability spaces $\{(X_k,\mathcal{B}_{X_k},\mu_k)\}_{k=1}^N$ and a cost function $c: X_1\times\cdots\times X_N\to \mathbb{R}$    we consider the Monge-Kantorovich problem
\begin{align}\label{MONKANT}
\inf_{\lambda\in\Pi(\mu_1,\ldots,\mu_N)}\int_{\prod_{k=1}^N X_k}c\,d\lambda.
\end{align}
Then for each ordered subset  $\mathcal{P}=\{i_1,\ldots,i_p\}\subsetneq\{1,...,N\}$ with $p\geq 2$ we create a new cost function $c_\mathcal{P}$ corresponding to the original cost function $c$ defined on $\prod_{k=1}^p X_{i_k}$.  This new cost function $c_\mathcal{P}$  enjoys many of the features of the original cost  $c$ while  it has the property that any optimal plan $\lambda$ of \eqref{MONKANT} restricted to $\prod_{k=1}^p X_{i_k}$ is also an  optimal plan  to the problem
\begin{align}\label{REDMONKANT}
\inf_{\tau\in\Pi(\mu_{i_1},\ldots\mu_{i_p})}\int_{\prod_{k=1}^p X_{i_k}}c_{\mathcal{P}}\,d\tau.
\end{align}
Our main contribution in this paper is to show that, 
for appropriate choices of index set $\mathcal{P}$,  one can recover the optimal plans of \eqref{MONKANT} from \eqref{REDMONKANT}. In particular, we study situations in which the problem \eqref{MONKANT} admits a unique solution depending on the uniqueness of the solution for  the lower marginal problems of the form  \eqref{REDMONKANT}. This allows us to prove many uniqueness results  for multi-marginal problems when the unique optimal plan is not necessarily  induced by a map. To this end, we extensively benefit from disintegration theorems and the $c$-extremality notions. Moreover, by employing this argument,  besides recovering many standard results on the subject including the pioneering work of Gangbo-\'Swi\c ech, several new  applications will be demonstrated to evince the  applicability of this argument.
\end{abstract}


\maketitle
\section{Introduction}
Our goal in  the present work is to study the uniqueness of optimal plans for  multi-marginal Monge-Kantorovich optimization problems. For an integer $N\geq 3$, let $\mathcal{N}=\{1,\ldots,N\}$ and consider a family of Borel probability spaces $\{(X_k,\mathcal{B}_{X_k},\mu_k)\}_{k=1}^N$. Let $P(X_1\times\cdots\times X_N)$ denote the set of Borel probability measures on $\prod_{k=1}^N X_k$ and $\Pi(\mu_1,\ldots,\mu_N)$ be its convex subset containing those elements with marginal equal to $\mu_k$ on $X_k$. It means that an element $\lambda\in P(X_1\times\cdots\times X_N)$ belongs to $\Pi(\mu_1,\ldots,\mu_N)$ if and only if
\begin{align*}
    \lambda (X_1\times\cdots\times X_{k-1}\times A_k\times X_{k+1}\times\cdots\times X_N)=\mu_k(A_k),\quad\forall A_k\in\mathcal{B}_{X_k}.
\end{align*}
It can be easily seen that $\Pi(\mu_1,\ldots,\mu_N)$ is a weak-star compact subset of $P(X_1\times\cdots\times X_N)$. For a cost function $c:\prod_{k=1}^N X_k\to\mathbb{R}$ the Monge-Kantorovich problem is as follows \cite{KANTOROVICH},
\begin{align*}\tag{MKP}\label{TAGMK}
    \inf\bigg\{\int_{\prod_{k=1}^N X_k} c\,d\lambda\ :\ \lambda\in\Pi(\mu_1,\ldots,\mu_N)\bigg\} .
\end{align*}
Existence of a minimizer of \eqref{TAGMK} is guaranteed whenever the cost function is non-negative and lower semi-continuous \cite{VILLANI2009}. This problem is in  fact a relaxation of the Monge problem which is initiated for the case $N=2$ in \cite{MONGE},
\begin{align*}\tag{MP}\label{TAGM}
    \inf\bigg\{\int_{X_1} c(x_1,T_2(x_1),\ldots,T_N(x_1))\,d\mu_1(x)\ :\ T_k:X_1\to X_k,\;\text{measurable},\ T_k\#\mu_1=\mu_k\bigg\} .
\end{align*}
Here, $T_k\#\mu_1=\mu_k$ indicates that the \textit{push-forward} of the measure $\mu_1$ through $T_k$ is $\mu_k$, i.e. $\mu_k(A_k)=\mu_1(T_k^{-1}(A_k))$, for $A_k\in \mathcal{B}_{X_k}$.
In contrary to Monge-Kantorovich problem, Monge problem may not have any solution. Furthermore, by \cite{KANTOROVICH,KELLERER} we know that the dual problem, by means of convex optimization, to \eqref{TAGMK} is  formulated as follows,
\begin{align*}\tag{DMKP}\label{TAGDMK}
    \sup\bigg\{\sum_{k=1}^N\int_{X_k} \phi_k\,d\mu_k\ :\ \phi_k\in L_1(X_k,\mu_k),\ \sum_{k=1}^N\phi_k\leq c \bigg\}.
\end{align*}
In the sequel, a solution to  \eqref{TAGMK} will be  called an \textit{optimal plan}, while a $(N-1)$-tuple $(T_2,\ldots,T_N)$ that is a solution to \eqref{TAGM} is said to be an \textit{optimal transport map}. Moreover, solutions to \eqref{TAGDMK} are of the form $(\phi_1,\ldots,\phi_N)$ where $\phi_k\in L_1(X_k,\mu_k)$ for $k\in\mathcal{N}$ and they are called \textit{potentials}.\\

The classical transport problem, the case $N=2$, has been studied tremendously. One approach for finding solutions has been done by means of derivatives of the cost function which is called \textit{twist} (or according to \cite{FATHIFIGALLI}, \textit{left twist}) condition. Precisely, if for every fixed $x_1$ the map $x_2\mapsto D_{x_1}c(x_1,x_2)$ is injective and the measure $\mu_1$ is absolutely continuous with respect to Lebesgue measure, one can expect the uniqueness of the solution to Monge problem, the unique solution $\lambda$ whose support is the graph of a unique optimal transport map $T$ \cite{CAFFARELLI,GANGBOTHESIS,GANGBOMCCANN,LEVIN,VILLANI2009}. Here, $D_{x_1}c$ denotes the derivative of the function $c$ with respect to the first variable $x_1$. Also, in \cite{AHMADKIMMCCANN}, by restricting the number of critical points of the cost function, another criterion, namely, \textit{sub-twist} condition, has been proposed and uniqueness of solution to \eqref{TAGMK} was proven. On top of that, in \cite{GANGBOMCCANN2,KIT}, the authors dealt with the sub-twisted case of Euclidean distance squared on the boundaries of (uniformly or strictly) convex domains and they obtained unique solution (either induced by a map or not). Moreover, in \cite{CHAMPIONDEPASCALE}, authors introduced \textit{strong twist} condition and obtained results in which the optimal plans are induced by measurable maps. Generally, the twist condition may fail in some cases, even when the underlying spaces are compact manifolds and the cost function is smooth, \cite{BIANCHINICARAVENNA,CHIAPPORIMCCANNNESHEIM}. This led to another attempt of finding new criteria to describe the  support of optimal plans which are named \textit{$m$-twist} and \textit{generalized-twist} conditions \cite{MOMENI-CHAR}. In these cases, it is obtained that the support of optimal plans are concentrated  on the union of several  graphs of measurable maps.

From another point of view, it is  known  that if \eqref{TAGMK} attains a unique solution then the unique solution has to be an extreme point of $\Pi(\mu_1,\ldots,\mu_N)$. Therefore, characterizing extreme points of this set is another aspect of research in this topic. Initially, for the case $N=2$, it is shown that a measure $\lambda$ is an extreme point of $\Pi(\mu_1,\mu_2)$ if and only if the space $L_1(X_1,\mu_1)+L_1(X_2,\mu_2)$ is dense in $L_1(X_1\times X_2,\lambda)$, \cite{DOUGLAS,LINDENSTRAUSS}. A more practical criterion for recognizing extreme points of $\Pi(\mu_1,\mu_2)$ was given in \cite{BENESSTEPAN} by means of \textit{aperiodic decomposition} of the support of the extreme measure. In \cite{MOMENI-DOUBLY}, equivalent to the aperiodic decomposition, the  author introduced the notion of maps with \textit{strongly disjointed graphs} from which one can get a characterization of the support of extreme points of $\Pi(\mu_1,\mu_2)$. On top of that, in \cite{MOMENI-RIFFORD} the notions of \textit{$c$-extreme} and \textit{$(c,P)$-extreme} sets have been developed which seem to be an  efficient tool  in order to study the  extreme points of $\Pi(\mu_1,\mu_2)$. These notions are the building blocks of our present work.\\

Another approach to optimal transport problems has been done in terms of decomposition of measures in $\Pi(\mu_1,\mu_2)$ by virtue of \textit{disintegration} theorems. Indeed, some authors took advantage of this perspective  to approach optimal transport problem in order to characterize the optimal measures \cite{BIANCHINICARAVENNA,FIGALLI}. For instance, in \cite{BIANCHINICARAVENNA}, authors studied situations in which a transference plan is extremal, optimal, and also  when  a unique transference plan  is concentrated on a specific set.\\

The multi-marginal case, ($N\geq 3$), appears in a vast variety of branches of sciences such as finance, economics, physics and electronics \cite{BEIGLBOCKLABORDPENKER,BUTTAZZODEPASCALEGORIGIORGI,CARLIEREKELAND,CHIAPPORIMCCANNNESHEIM,COTARFRIESECKEKLUPPELBERG,GALICHONLABORDTOUZI,GHOUSSOUBMOAMENI} and it has drawn attention of researchers recently. Various approaches have been taken to Monge and Monge-Kantorovich problems. In \cite{KIMPASS}, a general condition for Monge problem has been given in terms of twist condition on \textit{splitting} sets which is in fact a sort of generalization of twist condition for the case $N=2$. 
In \cite{MOAMENIPASS}, a modified version of the splitting sets jointly with $m$-twist and generalized-twist conditions have been employed to establish that the solutions of Monge-Kantorovich problem are concentrated  on several graphs. There are also several interesting works in which  uniqueness and non-uniqueness of solutions to Monge problem for specific  cost functions are addressed \cite{BINDINI,CARLIER,COLOMBODEPASCALEDIMARINO,GANGBOSWIECH,GELORINKAUSAMORAJALA,PASS1,PASS2,PASS3,PASSVARGAS}.\\

In the present paper, we study  the multi-marginal Monge-Kantorovich problems by analysing  a  reduced version of it. Indeed, for ordered subsets $\mathcal{P}=\{i_1,\ldots,i_p\}\subsetneq\mathcal{N}$ and $\mathcal{Q}=\mathcal{N}\backslash\mathcal{P}=\{j_1,\ldots,j_q\}$, we associate to  the cost function $c$  a new cost function $c_\mathcal{P}:\prod_{k=1}^p X_{i_k}\to\mathbb{R}$, by using the potentials $\{\phi_k\}_{k=1}^N$,  
\begin{align}\label{GEN-CP-function}
    c_\mathcal{P}(x_{i_1},\ldots,x_{i_p})=\inf_{\prod_{k=1}^q X_{j_k}}\big\{c(x_1,x_2,\ldots,x_N)-\sum_{k=1}^q\phi_{j_k}(x_{j_k})\big\}.
\end{align}
Correspondingly, we introduce the \textit{reduced Monge-Kantorovich} problem below
\begin{align*}\tag{RMKP}\label{TAGRMK}
    \inf\bigg\{\int_{\prod_{k=1}^p X_{i_k}} c_\mathcal{P}\,d\tau\ :\  \tau\in\Pi(\mu_{i_1},\ldots\mu_{i_p})\bigg\}.
\end{align*}
The cost function $c_\mathcal{P}$ inherits many properties of $c$, specifically, it will be shown that any  optimal plan of \eqref{TAGMK} restricted to   the space $\prod_{k=1}^p X_{i_k}$ is in fact an optimal plan of \eqref{TAGRMK}. Utilizing this fact, by appropriate choices for the subset $\mathcal{P}$, one can  conclude the uniqueness of the solution to \eqref{TAGMK} from the uniqueness of optimal plans of \eqref{TAGRMK}. More importantly, we determine the situations in which \eqref{TAGM} admits solutions depending on the case in which \eqref{TAGRMK} has solutions induced by measurable maps. Furthermore, we exhibit  circumstances under which \eqref{TAGMK} attains an unique optimal plan that  is not concentrated  on a graph of  a measurable  map. On top of that, necessary and sufficient conditions for the existence of the unique and non-unique solution to multi-marginal Monge problem will be investigated.\\

This paper is organized as follows: Section \ref{SectionPreliminaries} is devoted to the preliminaries and previously obtained results for two marginal optimization problems which will be extensively applied in the next sections. In Section \ref{SectionReduction}, we describe the reduction approach to optimization problems and describe the relation between an optimal plan of a multi-marginal problem and the reduced form of it by means of $c$-extremality notions. In  Section \ref{SectionDisintegration}, by benefiting from the disintegration of measures and the gluing lemma, we recover optimal plans of \eqref{TAGMK} from \eqref{TAGRMK}. Also, explicit form of optimal plans of the original multi-marginal problem will be constructed from its reduced version. Section \ref{SectionApplication} is designated to the applications of the  results established in the previous sections.  Finally, Section \ref{SectionAppendix} is where we recall many standard tools  that have been utilized throughout the paper.

\section{Preliminaries and the uniqueness in two marginal problems}\label{SectionPreliminaries}
To present our results, we need to recall some notations and preliminaries which will be frequently employed in the sequel together with some previously obtained results which are somehow crucial ingredients of this paper. We initiate this section by setting up the notations and basic conventions, then we state our definitions and theorems.\\

\noindent
For $N \geq 2$, let 
$\{(X_k,\mathcal{B}_{X_k},\mu_k)\}_{k=1}^N$ be a family of Borel probability measures on complete separable metric spaces. Depending on the number of marginals we may use different notations. For $N\leq 3$ we mainly use $(X_1,\mu_1)=(X,\mu)$, $(X_2,\mu_2)=(Y,\nu)$ and $(X_3,\mu_3)=(Z,\gamma)$ while for $N\geq 4$ we stick to the notation $(X_k,\mu_k)$.\\

\noindent
A set $\mathcal{N}=\{1,\ldots,N\}$ and its subset $\mathcal{P}=\{i_1,\ldots,i_p\}\subseteq\mathcal{N}$ will be considered as ordered subsets, that is,
\begin{align*}
i_1<i_2<\cdots<i_p.
\end{align*}
For any topological space $X$, the symbol $P(X)$ stands for the set of Borel probability measures on $X$. If $\lambda\in P(X)$, a set $A\subseteq X $ is said to be $\lambda$-negligible, if it is contained in a Borel set of zero $\lambda$-measure, or equivalently, we may say $X\backslash A$ has $\lambda$-full measure.\\

\noindent
We say $\lambda$ is \textit{concentrated} on a set $\mathcal{S}\subseteq X$, if $\mathcal{S}$ has a  $\lambda$-full measure. On top of that, $\text{Spt}(\lambda)$ denotes the \textit{support} of the measure $\lambda$ that is the smallest closed set on which $\lambda$ is concentrated.

\noindent
Furthermore, $\Pi(\mu_1,\ldots,\mu_N)\subseteq P(\prod_{k=1}^N X_k)$ is the convex subset of Borel probability measures with marginal $\mu_k$ on $X_k$ for $k\in\mathcal{N}$. Indeed, if the map $\pi_{k_0}:\prod_{k=1}^N X_k\to X_{k_0}$ is the natural projection map, then the push-forward of a measure $\lambda\in\Pi(\mu_1,\ldots,\mu_N)$ through $\pi_{k_0}$ is the measure $\mu_{k_0}$. In notation, 
\begin{align*}
    \pi_{k_0}\#\lambda=\mu_{k_0}.
    \end{align*}

\noindent
An element $\lambda\in\Pi(\mu_1,\ldots,\mu_N)$ is called an \textit{extreme} point if it can not be expressed as a convex combination of two distinct measures in $\Pi(\mu_1,\ldots,\mu_N)$, i.e. if there exist $\lambda_1,\lambda_2\in\Pi(\mu_1,\ldots,\mu_N) $ and $ \alpha\in (0,1)$ such that
\begin{align*}
\lambda=\alpha\lambda_1+(1-\alpha)\lambda_2, 
\end{align*}
then it implies that $\lambda=\lambda_1=\lambda_2$.\\

\noindent
Now we take the first step by describing the notion of $c$-extremality. Assume that $c:\prod_{k=1}^N X_k\to\mathbb{R}$ is a lower semi-continuous function such that there exist functions $a_k\in L_1(X_k,\mu_k)$ for $k\in\mathcal{N}$ and that 
\begin{align}\label{EXIS-INEQ}
    c(x_1,\ldots,x_N) \geq \sum_{k=1}^N a_k(x_k),\quad (x_1,\ldots,x_N)\in \prod_{k=1}^N X_k. 
    \end{align}
    The latter condition is to guarantee the  existence of a  solution to \eqref{TAGMK}. Here,  we always assume that  the $N$-marginal Monge-Kantorovich problem admits a solution, and consequently,  one may also consider different assumptions on the cost function $c$  to ensure the existence of at least one solution   to the $N$-marginal problem (see \cite{BEIGLBOCKSCHACHERMAYER,KELLERER,VILLANI2009}). 
Denote by $L_1(X_k,\mu_k)$ the Banach space of real-valued $\mu_k$-integrable functions on $X_k$.  Regarding the relation between \eqref{TAGMK} and \eqref{TAGDMK}, by \cite{VILLANI2009}, it is a well-known fact that there exist $\phi_k\in L_1(X_k,\mu_k)$, for $k\in\mathcal{N}$ which are solutions to \eqref{TAGDMK} together with a subset $\mathcal{S}\subseteq\prod_{k=1}^N X_k$ such that for any optimal plan $\lambda$ of \eqref{TAGMK}, we have
\begin{align}
    &\int_{\prod_{k=1}^N X_k} c\,d\lambda=\sum_{k=1}^N \int_{X_k} \phi_k\,d\mu_k,\label{SPLT-SET1}\\
    &\text{Spt}(\lambda)\subseteq\mathcal{S}\subseteq\bigg\{(x_1,\ldots,x_N) \in \prod_{k=1}^N X_k\ :\ c(x_1,\ldots,x_N)=\sum_{k=1}^N\phi_k(x_k) \bigg\},\label{SPLT-SET2}\\    
   & \sum_{k=1}^N \phi_k(x_k)\leq c(x_1,\ldots,x_N),\qquad (x_1,\ldots,x_N)\in\prod_{k=1}^N X_k.\label{SPLT-SET3}
\end{align}
 The aforementioned subset $\mathcal{S}$ is called a \textit{minimizing set} which depends on the function $c$ and marginals $\mu_k$. Additionally, it enjoys the following crucial property
\begin{align*}
\lambda\in\Pi(\mu_1,\ldots,\mu_N)\ \text{is an optimal plan if and only if}\ \lambda(\mathcal{S})=1.
\end{align*}
Another crucial feature of the set $\mathcal{S}$ is being $c$-\textit{cyclically monotone}, meaning that for any finite collection of the points
\begin{align*}
    \big\{(x^m_1,\ldots,x^m_N)\big\}_{m=1}^M\subseteq\mathcal{S},
\end{align*}
and for any permutation mapping $\sigma_k:\{1,\ldots,M\}\to\{1,\ldots,M\} $ the following inequality holds
\begin{align*}
    \sum_{m=1}^M c(x^m_1,\ldots,x^m_N)\leq\sum_{m=1}^M c(x^{\sigma_1(m)}_1,\ldots,x^{\sigma_N(m)}_N).
\end{align*}
Naturally,  a measure $\lambda\in\Pi(\mu_1,\ldots,\mu_N)$ is said to be $c$-cyclically monotone if it is concentrated on a $c$-cyclically monotone set. In our setting, a measure $\lambda$ is an optimal plan if and only if it is $c$-cyclically monotone (see \cite{VILLANI2009}). To see more about $c$-cyclical monotonicity  in the multi-marginal case and its related optimality consequences one can refer to  \cite{KAUSAMO2,KAUSAMO,TLIM}. 

\begin{remark}
It should be mentioned that in \cite{BEIGLBOCKLABORDPENKER,KELLERER,VILLANI2009} there have been several cases of dual attainment and the conditions under which there is no gap between \eqref{TAGMK} and \eqref{TAGDMK}. Particularly, in \cite{KELLERER}, the multi-marginal case has been studied and it is obtained that in the case that the cost function $c$ is measurable which satisfies \eqref{EXIS-INEQ} for some $a_k\in L_1(X_k,\mu_k)$ for $k\in\mathcal{N}$ and $X_i$'s are Polish spaces there is no gap between infimum in \eqref{TAGMK} and supremum in \eqref{TAGDMK} (see \cite[Theorem 2.14 and Corollary 2.18]{KELLERER}). Moreover, if $c$  is lower semi-continuous the minimizer for \eqref{TAGMK} exists (see \cite[Theorem 2.19]{KELLERER}). On top of that, if $c$ is such that the supremum in \eqref{TAGDMK} is finite and that \eqref{EXIS-INEQ} holds, then the maximizer for \eqref{TAGDMK} exists (see \cite[Theorem 2.21]{KELLERER}). Therefore, for the case that the function $c$ is continuous and Borel probability measures $\mu_k$'s are compactly supported or the domains $X_k$'s are closed smooth manifolds, there is no gap and attainment of infimum and supremum of \eqref{TAGMK} and \eqref{TAGDMK} are guaranteed.
\end{remark}

\subsection{The two marginal case}
Here,  we recall the notion of $c$-extremality  from \cite{MOMENI-RIFFORD}. Consider Borel probability spaces $(X,\mathcal{B}_X,\mu)$ and $(Y,\mathcal{B}_Y,\nu)$. In this case, the set $\mathcal{S}$ in \eqref{SPLT-SET2} is contained in the following set
\begin{align*}
\bigg\{ (x,y)\in X\times Y\ :\ c(x,y)=\phi_1(x)+\phi_2(y) \bigg\}.
\end{align*}
Consider the following set-valued functions
\begin{align*}
    &F_\mathcal{S}:X\to 2^Y,\quad F_\mathcal{S}(x)=\{y\in Y\ :\ (x,y)\in \mathcal{S}\},\\
    &f_{\mathcal{S},c}:X\to 2^Y,\quad f_{\mathcal{S},c}(x)=\argmax\{c(x,y)\ :\ y\in F_\mathcal{S}(x)\}.
\end{align*}
Domains of these functions are defined as
\begin{align*}
    \text{Dom}(F_\mathcal{S})=\{x\ :\ F_\mathcal{S}(x)\neq \emptyset\},\quad \text{Dom}(f_{\mathcal{S},c})=\{x\ :\ f_{\mathcal{S},c}(x)\neq \emptyset\}.
\end{align*}
    Put
\begin{align*}
    D(\mathcal{S},c)=\{x\in X\ : f_{\mathcal{S},c}(x)\text{ is a singleton}\}.
\end{align*}
We give the definition of a $c$-extreme minimizing set.
\begin{definition}\label{DEF-CEXT}
    Say that a minimizing set $\mathcal{S}$ is $c$-extreme, if there exist $\mu$ and $\nu$-full measure sets $X_0$ of $X$ and $Y_0$ of $Y$ respectively  such that the two following conditions hold
    \begin{align*}
        &(i)\; \text{Dom}(F_\mathcal{S})\cap X_0=\text{Dom}(f_{\mathcal{S},c})\cap X_0,\\
        &(ii)\; \forall x_1\neq x_2\in X_0 ,\quad \bigg(F_{\mathcal{S}}(x_1)\backslash\{y_1\}\bigg)\cap \bigg(F_{\mathcal{S} }(x_2)\backslash\{y_2\}\bigg)\cap Y_0=\emptyset,\quad\forall y_i\in f_{\mathcal{S},c}(x_i),\ i=1,2. 
    \end{align*}
\end{definition}

\begin{remark}
    It should be remarked that the celebrated twist and sub-twist conditions can be extracted as special cases of $c$-extremality. In fact, in the case that the $\mu$-full measure set $X_0$ is contained in $D(\mathcal{S},c)$, then for all $x\in X_0$ the set $f_{\mathcal{S},c}(x)$ is singleton and these two notions can be derived. Precisely, by \cite{VILLANI2009}, we say that the function $c$ satisfies twist condition, if for fixed $x\in X$, the map $y\mapsto D_xc(x,y)$ is injective. By means of $c$-extremality and noting the fact that for all $x\in X_0$ the set $f_{\mathcal{S},c}(x)$ is singleton, this condition reads as
    \begin{align*}
        F_{\mathcal{S} }(x)\backslash f_{\mathcal{S},c}(x)=\emptyset,\quad\forall x\in X_0.
    \end{align*}
    Additionally, by \cite{CHIAPPORIMCCANNNESHEIM}, the sub-twist condition is that for each $y_1\neq y_2 \in Y$ the map $x \mapsto c(x, y_1) -c(x, y_2)$, defined on $X$, has no critical points, save at most one global minimum and at most one global maximum. In terms of $c$-extremality, we have
    \begin{align*}
        \bigg(F_{\mathcal{S} }(x_1)\backslash f_{\mathcal{S},c}(x_1)\bigg)\cap F_{\mathcal{S} }(x_2)=\emptyset ,\quad \forall x_1\neq x_2\in X_0.
    \end{align*}
\end{remark}
Also, the notion of $c$-extremality has been developed for the situations in which the space $Y$ has measurable partitions, and consequently, the definition of $(c,P)$-extremality has been obtained \cite{MOMENI-RIFFORD}. In details, let $P=\{Y_i\}_{i\in\mathbb{I}}$ with $\mathbb{I}\subseteq\mathbb{N}$ be a Borel ordered partition for $Y$ meaning that $Y_i$'s are pairwise disjoint Borel sets with union equal to $Y$. Consider the map $\iota:\text{Dom}(F)\to\mathbb{N}$ with the following definition
\begin{align*}
    \iota(x)=\min\{i\in\mathbb{I}\ : F_{\mathcal{S}}(x)\cap Y_i\neq \emptyset\}.
\end{align*}
We define the corresponding set-valued function $f_{\mathcal{S},c,P}:X\to 2^Y$ by
\begin{align*}
    f_{\mathcal{S},c,P}(x)=\argmax\{c(x,y)\ :\ y\in F_{\mathcal{S}}(x)\cap Y_{\iota(x)}\},
\end{align*}
together with the following set
\begin{align*}
    D(\mathcal{S}, c,P)=\{x\in X\ : f_{\mathcal{S},c,P}(x)\; \text{is a singleton}\}.
\end{align*}
Similarly, we have $\text{Dom}(f_{\mathcal{S},c,P})=\{x\in X\ :\ f_{\mathcal{S},c,P}(x)\neq \emptyset\}$. Now, we state the definition of $(c,P)$-extreme minimizing sets which is indeed a local version of $c$-extreme minimizing sets, i.e. when $P=\{Y\}$ these two definitions coincide.
   \begin{definition}
   A measure $\lambda$ is called $(c,P)$-extreme, if there exist $\mu$ and $\nu$-full measure sets $X_0$ of $X$ and $Y_0$ of $Y$ respectively such that
    \begin{align*}
        &(i)\; \text{Dom}(F_{\mathcal{S} })\cap X_0=\text{Dom}(f_{\mathcal{S},c,P})\cap X_0,\\
        &(ii)\; \forall x_1\neq x_2\in X_0 ,\quad \bigg(F_{\mathcal{S}}(x_1)\backslash\{y_1\}\bigg)\cap \bigg(F_{\mathcal{S}}(x_2)\backslash\{y_2\}\bigg)\cap Y_0=\emptyset,\quad \forall y_i\in f_{\mathcal{S},c,P}(x_i),\ i=1,2. 
    \end{align*}
\end{definition}
By employing these two notions of $c$-extremality the following theorem is established (see \cite[Theorem 2.10 and Theorem 2.12]{MOMENI-RIFFORD}).
\begin{theorem}\label{THc-EXT}
    Let $X$ and $Y$ be smooth closed manifolds equipped with probability measures $\mu$ and $\nu$, respectively. Assume that $c:X\times Y\to \mathbb{R}$ is a continuous function and $\mathcal{S}\subseteq X\times Y$ is a minimizing set. If
    \begin{enumerate}
        \item\label{THc-EXT1} $\mathcal{S}$ is $c$-extreme, or
        \item\label{THc-EXT2} $P=\{Y_i\}_{i\in\mathbb{I}}$ is a Borel ordered partition for $Y$ such that $\mathcal{S}$ is $(c,P)$-extreme,
    \end{enumerate}
   then there exists a unique $\lambda\in\Pi(\mu,\nu)$ with $\lambda(\mathcal{S})=1$.
\end{theorem}

\begin{remark}\label{REM-NOTICE}
\begin{enumerate}
    \item \label{REM-NOTICE1}
A generalizations of the Theorem \ref{THc-EXT} is also given in \cite{MOMENI-RIFFORD}, as well. In detail, in the case that $X$ and $Y$ are Polish spaces, the non-negative cost function $c$ is measurable, lower semi-continuous, and $\mu\otimes\nu$-a.e. finite, it is established that existence of a minimizing set which is $c$-extreme guarantees existence of the unique solution.
\item\label{REM-NOTICE2}
It is crucial to note that same results in Theorem \ref{THc-EXT} can be obtained when we choose $\mathcal{S}=\text{Spt}(\lambda)$, where $\lambda$ is an arbitrary optimal plan. In this case, we say the optimal measure $\lambda$ is $c$-extreme, and we may use the notations $F_{\lambda}$ and $f_{\lambda ,c}$.
\item\label{REM-NOTICE3} In fact, in the proof of the Theorem \ref{THc-EXT} it has been shown that the unique optimal $\lambda$ is an extreme point of $\Pi(\mu,\nu)$. We will benefit from this fact in our proofs.
\end{enumerate}
\end{remark}
Additionally, by taking advantage of the notion of sub-twist condition, another result regarding the uniqueness of the solution of \eqref{TAGMK} has been obtained \cite[Theorem 5.2]{AHMADKIMMCCANN}.
\begin{theorem}
    Let $(X,\mathcal{B}_X,\mu)$ and $(Y,\mathcal{B}_Y,\nu)$ be Borel probability spaces with complete separable manifolds $X$ and $Y$, and $\mu$ absolutely continuous in each coordinate chart on $X$. Assume that $c:X\times Y\to\mathbb{R}$ is a bounded continuously differentiable cost function such that it satisfies the sub-twist condition. Moreover, let $D_xc(x, y)$ be locally bounded in $x$, uniformly in $Y$ . Then the associated Monge-Kantorovich problem admits a unique solution $\lambda\in\Pi(\mu,\nu)$ that is concentrated on the disjoint union of the graph and anti-graph of two maps.
\end{theorem}
\subsection{Optimal plans not induced by a single map}
We shall also consider cases at which the Monge-Kantorovich problem has a solution concentrating on the graph of several maps. To this end, 
 we recall some preliminaries from \cite{MOMENI-CHAR}.  First, we state the following definition which can be seen as a generalization of the well-known twist condition in optimal transportation. 
\begin{definition}\label{DEF-2-TWIST}
Let $c:X\times Y\to\mathbb{R}$ be a function for which $D_xc$ exists on its domain. For a fixed $(x_0,y_0)\in X\times Y$, consider the following set
\begin{align*}
    L(x_0,y_0)=\{y\in Y\ :\ D_xc(x_0,y)=D_xc(x_0,y_0)\}\subseteq Y.
\end{align*}
We say that $c$ satisfies
\begin{enumerate}
    \item $m$-twist condition, if for all $(x_0,y_0)\in X\times Y$ the subset $L(x_0,y_0)$ has at most $m$ elements,
    \item generalized-twist condition, if for all $(x_0,y_0)\in X\times Y$ the subset $L(x_0,y_0)$ is finite.
\end{enumerate}
\end{definition}
These definitions lead us to the existence of measures which are concentrated on the graph of several maps. In the following $C_b(X\times Y)$ denotes the space of all real-valued bounded continuous functions on $X\times Y$.
\begin{definition}\label{DEF-MULTI-GRAPH}
A measure $\lambda\in\Pi(\mu,\nu)$ is said to be concentrated on the union of the graphs of measurable maps $\{T_k\}_{k=1}^m$ with $T_k:X\to Y$, if there exists a sequence of measurable maps $\alpha_k :X\to [0,1]$, $k=1,\ldots,m$ such that for $\mu$-a.e. $x\in X$ we have $\sum_{k=1}^m\alpha_k(x)=1$ and that 
\begin{align*}
    \int_{X\times Y}f(x,y)\,d\lambda(x,y)=\sum_{k=1}^m\int_X \alpha_k(x) f(x,T_k(x))\,d\mu(x),\quad \forall f\in C_b(X\times Y).
\end{align*}
In notation, we write $\lambda=\sum_{k=1}^m (\text{id}\times T_k)\#(\alpha_k\mu_k)$.
\end{definition}

\begin{remark}
    It should be remarked that a measure $\lambda$ in Definition \ref{DEF-MULTI-GRAPH} admits the following disintegration
\begin{align*}
    \lambda=\big(\sum_{k=1}^m\alpha_k(x)\delta_{T_k(x)}\big)\otimes\mu,
\end{align*}
where $\delta_{T_k(x)}$ denotes the dirac measure at the point $T_k(x)\in Y$ which belongs to $P(Y)$.
\end{remark}

Now, we are prepared to state the following theorem (see \cite[Theorem 1.3]{MOMENI-CHAR}).
\begin{theorem}\label{MOMEN-CHAR-TH}
Let $(X,\mathcal{B}_X,\mu)$ and $(Y,\mathcal{B}_Y,\nu)$ be complete separable Riemannian manifold and Polish spaces, respectively. 
    Assume that the bounded continuous cost function $c$ satisfies the $m$-twist condition, the measure $\mu$ is non-atomic and is such that every $c$-concave function is $\mu$-almost everywhere differentiable on its domain. Then each optimal plan $\lambda\in \Pi(\mu,\nu)$ of Monge-Kantorovich problem associated with $c$ is concentrated on the union of graphs of at most $m$ maps.
\end{theorem}

\begin{example}[Gromov-Wasserstein problem]
    Let the spaces $X$ and $Y$ be subspaces of $\mathbb{R}^n$ equipped with the probability measures $\mu$ and $\nu$ where $\mu$ is non-atomic and consider the following cost function
    \begin{align*}
        c:X\times Y\to \mathbb{R},\quad c(x,y)=|x|^2|y|^2+\xi \langle Ax,y\rangle,
    \end{align*}
    where $\xi$ is non-zero parameter in $\mathbb{R}$ and $A$ is an $n\times n$ invertible matrix.
    Then the function $c$ satisfies $2$-twist condition. In fact, if $y\in Y$ is such that $D_xc(x_0,y)=D_xc(x_0,y_0)$ then
    \begin{align}\label{GRO-WAS-DEF-EQ}
        y =\frac{2}{\xi}(A^{\top})^{-1}x_0(|y_0|^2-|y |^2)+y_0.
    \end{align}
    Evidently, one solution to the equation \eqref{GRO-WAS-DEF-EQ} is $y_0$. Note that if $y$ is a solution to \eqref{GRO-WAS-DEF-EQ} then it lives on the straight line with the direction vector $(A^{\top})^{-1}x_0$ passing through the point $ y_0$. We shall show that this equation has at most one solution which is different from $y_0$. Let there be two distinct solutions $y_1$ and $y_2$ for \eqref{GRO-WAS-DEF-EQ} such that $y_i\neq y_0$, for $i=1,2$. Then since $y_i$'s are distinct points on the straight line then one of them can be written as a convex combination of two other points. Without loss of generality assume that
    \begin{align}\label{GRO-WAS-LIN}
        y_1=\theta y_0 +(1-\theta) y_2, \quad\text{for some}\; \theta\in (0,1).
    \end{align}
Substituting $y_1$ and $y_2$ from \eqref{GRO-WAS-DEF-EQ} into \eqref{GRO-WAS-LIN} and some cancellation yield that
\begin{align*}
    (\theta |y_0|+(1-\theta )|y_2|^2-|y_1|^2)x_0=0.
\end{align*}
If $x_0\neq 0$, then we obtain that
\begin{align*}
    |y_1|^2=\theta |y_0|+(1-\theta )|y_2|^2.
\end{align*}
On the other hand, from \eqref{GRO-WAS-LIN} and strict convexity of the Euclidean  norm, we obtain the following contradiction
\begin{align*}
    |y_1|^2=|\theta y_0 +(1-\theta) y_2|^2<\theta |y_0|+(1-\theta )|y_2|^2=|y_1|^2,
\end{align*}
which shows that $y_1=y_2$.\\
If $x_0=0$, then from \eqref{GRO-WAS-DEF-EQ} we have $y_1=y_2=y_0$. Hence, if $\mu$ is non-atomic then the optimal plan of 
    \begin{align}\label{GRO-WAS-OPT-PROB}
        \min\bigg\{\int_{X\times Y} -(|x|^2|y|^2+\xi \langle Ax,y\rangle)d\lambda\ : \ \lambda\in \Pi(\mu,\nu)\bigg\},
    \end{align}
    is concentrated on the graph of two measurable maps. The problem \eqref{GRO-WAS-OPT-PROB} appears in a special case of the Gromov-Wasserstein problem initiated in \cite{MEMOLI} which is indeed an extension of Gromov-Hausdorff distance to the probability measure space concept \cite{GROMOV}. In more detail, when optimization problem is
\begin{align*}
    \min\bigg\{\int_{X\times Y}\int_{X\times Y}\big||x-x'|^2-|y-y'|^2\big|^2\,d\lambda(x,y)\,d\lambda(x',y')\ :\ \lambda\in \Pi(\mu,\nu)\bigg\},
\end{align*}
then this problem reduces to \eqref{GRO-WAS-OPT-PROB} for which we have only considered the case where $A$ is invertible.
    \end{example}

\section{A reduction argument}\label{SectionReduction}
 For an integer $N\geq 3$, let $\mathcal{N}=\{1,\ldots,N\}$ and consider a family of Borel probability spaces $\{(X_k,\mathcal{B}_{X_k},\mu_k)\}_{k=1}^N$.   For a cost function $c:\prod_{k=1}^N X_k\to\mathbb{R}$ the Monge-Kantorovich problem is as follows, 
\begin{align*}\tag{MKP}\label{TAGMK}
    \inf\bigg\{\int_{\prod_{k=1}^N X_k} c\,d\lambda\ :\ \lambda\in\Pi(\mu_1,\ldots,\mu_N)\bigg\} .
\end{align*}
The dual problem, by means of convex optimization, to \eqref{TAGMK} is  formulated as follows,
\begin{align*}\tag{DMKP}\label{TAGDMK}
    \sup\bigg\{\sum_{k=1}^N\int_{X_k} \phi_k\,d\mu_k\ :\ \phi_k\in L_1(X_k,\mu_k),\ \sum_{k=1}^N\phi_k\leq c \bigg\}.
\end{align*}
For an ordered subset \[\mathcal{P}=\{i_1,\ldots,i_p\}\subsetneq\mathcal{N},\]  with $p\geq 2$ and $\mathcal{Q}=\mathcal{N}\backslash\mathcal{P}=\{j_1,\ldots,j_q\}$, we associate to  the cost function $c$  a new cost function $c_\mathcal{P}:\prod_{k=1}^p X_{i_k}\to\mathbb{R}$ as follows,
\begin{align}\label{GEN-CP-function}
    c_\mathcal{P}(x_{i_1},\ldots,x_{i_p})=\inf_{\prod_{k=1}^q X_{j_k}}\big\{c(x_1,x_2,\ldots,x_N)-\sum_{k=1}^q\phi_{j_k}(x_{j_k})\big\}.
\end{align}
where  $\{\phi_1,...,\phi_N\}$ is a solution of \ref{TAGDMK}.
Correspondingly, we introduce the \textit{reduced Monge-Kantorovich} problem below
\begin{align*}\tag{RMKP}\label{TAGRMK}
    \inf\bigg\{\int_{\prod_{k=1}^p X_{i_k}} c_\mathcal{P}\,d\tau\ :\  \tau\in\Pi(\mu_{i_1},\ldots\mu_{i_p})\bigg\}.
\end{align*}

The main motivation of this section is to investigate  the relation between optimal plans of \eqref{TAGMK} and \eqref{TAGRMK},  and to obtain a criterion for the uniqueness of  \eqref{TAGMK}. Here,  we always assume that  the infimum in \eqref{TAGMK} is attained.  To this end,  by demonstrating a connection between Kantorovich dual problems to \eqref{TAGMK} and \eqref{TAGRMK}, we provide a relation between their solutions in Proposition \ref{PROP-3}. Afterwards, in Theorems \ref{MAINTHEOREM} and \ref{MAINTHEOREMCP}, by choosing appropriate index subsets $\mathcal{P}$ in \eqref{GEN-CP-function}, we define a collection of cost functions $c_j$ for $j\in\mathcal{N}\backslash\{1\}$ and we consider their corresponding reduced Monge-Kantorovich problems. We prove uniqueness of the solution to \eqref{TAGMK} based on the uniqueness of the solution to these particular problems. In fact, in order to establish the desired uniqueness, the pivotal ideas are suitable choices of $\mathcal{P}$ and the uniqueness of the solution to \eqref{TAGRMK} associated with these $c_j$'s. To provide the latter condition, we benefit from the notion of $c$-extremality  described in Remark \ref{REM-NOTICE}. \\

In the first step, we obtain an optimal plan of \eqref{TAGRMK} from \eqref{TAGMK}. Indeed, the following proposition  shows  that the existence of the solutions to the original $N$-marginal Monge-Kantorovich problem guarantees the existence of the solutions to the reduced version of it, for all non-empty subsets $\mathcal{P}$ of $\mathcal{N}$.
\begin{proposition}\label{PROP-3}
Let $\mathcal{P}=\{i_1,\ldots,i_p\}$ be  any ordered subset of $\mathcal{N}$ and  $\pi_\mathcal{P}:\prod_{k=1}^N X_k\to \prod_{k=1}^p X_{i_k}$ be the corresponding projection map. If the family of Borel probability spaces $\{(X_k,\mathcal{B}_{X_k},\mu_k)\}_{k=1}^N$ is such that $\lambda\in\Pi(\mu_1,\ldots,\mu_N)$ is an optimal plan for \eqref{TAGMK}, and \eqref{TAGDMK} admits a solution, then the measure $\lambda_\mathcal{P} \in \Pi(\mu_{i_1},\ldots\mu_{i_p})$  defined by  $\lambda_\mathcal{P}=\pi_\mathcal{P}\#\lambda$ is an optimal plan for \eqref{TAGRMK},and the dual problem to \eqref{TAGRMK} admits a solution.
\end{proposition}
    \begin{proof}
We first  show that $\{\phi_{i_k}\}_{k=1}^{p}$ solves the dual problem to \eqref{TAGRMK}, where $\{\phi_k\}_{k=1}^N$ is the solution to \eqref{TAGDMK}. Let $\mathcal{Q}=\mathcal{N}\backslash\mathcal{P}=\{j_1,\ldots,j_q\}$. From \eqref{SPLT-SET3}, we have
\begin{align*}
\sum_{k=1}^{p}\phi_{i_k}(x_{i_k})\leq c(x_1,\ldots,x_N)-\sum_{k=1}^{q}\phi_{j_k}(x_{j_k}), \qquad \forall (x_1,\ldots,x_N) \in \prod_{k=1}^N X_k,
\end{align*}
and consequently, for all $(x_1,\ldots,x_N) \in \prod_{k=1}^N X_k$, we have
\begin{align}\label{INEQ}
\sum_{k=1}^{p}\phi_{i_k}(x_{i_k})\leq c_\mathcal{P}(x_{i_1},\ldots,x_{i_p})\leq  c(x_1,\ldots,x_N)-\sum_{k=1}^{q}\phi_{j_k}(x_{j_k}).
\end{align}
Integrating all sides of the inequalities in \eqref{INEQ} with respect to $d\lambda$, and using marginality of $\mu_k$'s and $\lambda_\mathcal{P}$, it is obtained that
\begin{align*}
\sum_{k=1}^{p}\int_{X_{i_k}}\phi_{i_k}(x_{i_k})\,d\mu_{i_k}\leq \int_{\prod_{k=1}^{p}X_{i_k}}c_\mathcal{P}(x_{i_1},\ldots,x_{i_p})\,d\lambda_{\mathcal{P}}\leq \int_{\prod_{k=1}^N X_k} c(x_1,\ldots,x_N)\,d\lambda-\sum_{k=1}^q\int_{X_{j_k}}\phi_{j_k}(x_{j_k})\,d\mu_{j_k}.
\end{align*}
Now, by applying \eqref{SPLT-SET1}, we conclude the desired equality as follows
\begin{align*}
\sum_{k=1}^{p}\int_{X_{i_k}}\phi_{i_k}(x_{i_k})\,d\mu_{i_k}=\int_{\prod_{k=1}^{p}X_{i_k}}c_\mathcal{P}(x_{i_1},\ldots,x_{i_p})\,d\lambda_{\mathcal{P}}.
\end{align*}
This shows that $\lambda_\mathcal{P}$ is a solution to \eqref{TAGRMK} as well as $\{\phi_{i_k}\}_{k=1}^p$ is a solution to the dual problem to \eqref{TAGRMK}.
    \end{proof}
\begin{remark}
We would like to remark that an implicit version of the above result for specific cost functions have already appeared in  \cite{PASSVARGAS2} and \cite{CARLIER} where the authors prove the existence of Monge type solutions. Indeed, in \cite{CARLIER} this method is applied for strictly monotone functions of order $2$ and in \cite{PASSVARGAS2} it is done for the cyclic cost function. Moreover, in \cite{KAUSAMO}, the authors investigated $c$-splitting set in $\prod_{k=1}^NX_k$ and its project on $X_i\times X_j$ and showed its properties when the cost function $c$ can be written as the sum of cost functions $c_{ij}:X_i\times X_j\to\mathbb{R}$.
\end{remark}
In the next attempt, we derive uniqueness of the optimal plan of \eqref{TAGMK} from special form of \eqref{TAGRMK}. To this end, let us consider \eqref{TAGRMK} with $\mathcal{P}_j=\{1,\ldots,j\}$, for $j\in\mathcal{N}\backslash\{1\}$. Consequently, the optimization problem \eqref{TAGRMK} turns into
\begin{align}\label{GEN-TRI}
    \inf_{\gamma\in\Pi(\mu_1,\ldots,\mu_j)}\int c_j(x_1,\ldots,x_j)d\gamma, \quad j\in\mathcal{N}\backslash\{1\} .
\end{align}
where from \eqref{GEN-CP-function}, we have
\begin{align}\label{GEN-CJ-function}
    c_j(x_1,\ldots,x_j)=\inf_{\prod_{k=j+1}^N X_k}\bigg\{c(x_1,\ldots,x_N)-\sum_{k=j+1}^N\phi_k(x_k)\bigg\},  \quad j\in\mathcal{N}\backslash\{1\}.
\end{align}
Note that the function $c$ can be viewed as a function on $Z_N\times X_N$ where $Z_N=\prod_{k=1}^{N-1}X_k$. Moreover, we may consider the function $c_j$ as a  function on $Z_j\times X_j$ where  we assume $Z_j=\prod_{k=1}^{j-1}X_k$. For each $j\in\mathcal{N}\backslash\{1,N\}$, we  denote by \[\lambda_j:=\pi_{\mathcal{P}_j}\#\lambda,\] the restriction  of $\lambda\in \Pi(\mu_1,\ldots,\mu_N)$ on $\prod_{k=1}^{j}X_k$ for which we have that $\lambda_j \in \Pi(\mu_1,\ldots,\mu_j)$. Note that,  we may also  consider $\lambda_j$ as a measure on $Z_j \times X_j$ which belongs to  $\Pi(\lambda_{j-1},\mu_{j})$  in which we have set $\lambda_1:=\mu_1$.\\

\noindent
To prove the next theorem, we use the terminology given in Remark \ref{REM-NOTICE}-\eqref{REM-NOTICE2}. It should be clarified that when we say a measure $\lambda$ is $c$-extreme it means that it is $c(z_N,x_N)$-extreme, i.e. the function $c$ is considered as a function on $Z_N\times X_N$. This implies that the associated set-valued maps $F_{\lambda}$ and $f_{\lambda,c}$ are maps from $Z_N$ to $2^{X_N}$. Similarly, we may say a measure (e.g. the marginal $\lambda_j$) is $c_j$-extreme and we mean it is $c_j(z_{j},x_{j})$-extreme. Same holds true for associated set-valued maps. Note that for $\lambda_{2}\in \Pi(\mu_1,\mu_2)$, being $c_{2}$-extreme is clear by the Definition \ref{DEF-CEXT}. We are now ready to state  our main result in this section.

\begin{theorem}\label{MAINTHEOREM}
Let  $\{(X_k,\mathcal{B}_{X_k},\mu_k)\}_{k=1}^N$ be a family of Borel probability spaces such that $X_k$'s are Polish spaces, $c:\prod_{k=1}^NX_k\to[0,\infty]$, is lower semi-continuous, $\otimes_{k=1}^N\mu_k$-a.e. finite, and there exists a finite transport plan. If each optimal plan of the problem \eqref{TAGMK} is $c$-extreme on $Z_N\times X_N$ and is such that its restrictions $\lambda_j=\pi_{\mathcal{P}_j}\#\lambda$ on $Z_j\times X_j$ are $c_j$-extreme, for $j\in\mathcal{N}\backslash\{1,N\}$, then \eqref{TAGMK} admits a unique solution.
\end{theorem}

The following proposition shows the link between extremality of a measure in $\Pi(\mu_1,\ldots,\mu_N)$ and its restriction $\lambda_j$ on $Z_j\times X_j$, for $j\in\mathcal{N}\backslash\{1,N\}$ which is the pivotal factor in the proof of Theorem \ref{MAINTHEOREM}.

\begin{proposition}\label{PROP-EXTREME}
    Let $\lambda\in \Pi(\mu_1,\ldots,\mu_N)$ and $\lambda_j$ be its restriction on $Z_j\times X_j$, for $j\in\mathcal{N}\backslash\{1,N\}$. If $\lambda$ is an extreme point of $\Pi(\lambda_{N-1},\mu_N)$ and $\lambda_j$ is an extreme point of $\Pi(\lambda_{j-1},\mu_j)$ for $j\in\mathcal{N}\backslash\{1,N\}$, then $\lambda$ is an extreme point of $\Pi(\mu_1,\ldots,\mu_N)$.
    \begin{proof}
        We shall show that $\lambda$ is an extreme point of $\Pi(\mu_1,\ldots,\mu_N)$. Let it be otherwise, and let there exist two distinct measures $\alpha$ and $\beta$ in $\Pi(\mu_1,\ldots,\mu_N)$ such that
\begin{align}\label{QQQ}
    \lambda=\frac{1}{2}(\alpha+\beta),\qquad \lambda\neq\alpha ,\beta .
\end{align}
Let $\alpha_j=\pi_{\mathcal{P}_j}\#\alpha$ and $\beta_j=\pi_{\mathcal{P}_j}\#\beta$ be restrictions of $\alpha$ and $\beta$ on $Z_j\times X_j$. Then clearly we have
\begin{align*}
    \lambda_j=\frac{1}{2}(\alpha_j+\beta_j),\quad j\in\mathcal{N}\backslash\{1,N\}.
\end{align*}
Since, $\lambda_{2}$ is an extreme point of $\Pi(\mu_1,\mu_2)$, then it is concluded that
\begin{align*}
    \alpha_{2}=\beta_{2}=\lambda_{2},\quad\text{and}\quad \alpha_{3},\;\beta_{3}\in\Pi(\lambda_{2},\mu_3).
\end{align*}
Again, since 
$\lambda_{3}$ is an extreme point of $\Pi(\lambda_{2},\mu_3)$, then we obtain that
\begin{align*}
    \alpha_{3}=\beta_{3}=\lambda_{3},\quad\text{and}\quad \alpha_{4},\;\beta_{4}\in\Pi(\lambda_{3},\mu_4).
\end{align*}
By repeating the same process, we reach to the point that
\begin{align*}
    \alpha_{N-1}=\beta_{N-1}=\lambda_{N-1},\quad\text{and}\quad \alpha,\;\beta\in\Pi(\lambda_{N-1},\mu_N).
\end{align*}
Finally, the extremality of $\lambda$ in $\Pi(\lambda_{N-1},\mu_N)$ and \eqref{QQQ} lead to a contradiction. Therefore, $\lambda$ is an extreme point of $\Pi(\mu_1,\ldots,\mu_N)$.
    \end{proof}
\end{proposition}

\noindent
\textbf{Proof of Theorem \ref{MAINTHEOREM}.}
The proof is actually based on the case $N=2$. Indeed, since an optimal plan $\lambda$ of \eqref{TAGMK} is $c$-extreme then by Theorem \ref{THc-EXT}-\eqref{THc-EXT1} (and the Remark \ref{REM-NOTICE}-\eqref{REM-NOTICE3} afterwards) it is an extreme point of $\Pi(\lambda_{N-1},\mu_N)$. We note that by Remark \ref{REM-NOTICE}-\eqref{REM-NOTICE1} we are allowed to apply Theorem \ref{THc-EXT}. Moreover, by Proposition \ref{PROP-3}, we know that $\lambda_j$ is an optimal plan for \eqref{GEN-TRI}, which by assumption is $c_j$-extreme. Hence, by similar argument it is an extreme point of $\Pi(\lambda_{j-1},\mu_{j})$. Now, by applying Proposition \ref{PROP-EXTREME} we obtain that the optimal plan $\lambda$ is an extreme point of $\Pi(\mu_1,\ldots,\mu_N)$ which yields the desired uniqueness. This is because of the fact that if there is another optimal plan $\lambda'$ then it is an extreme ponit as well. Moreover, any convex combination of $\lambda$ and $\lambda'$, say $\lambda_t=t\lambda+(1-t)\lambda'$, for $t\in (0,1)$, is an optimal plan and it has to be an extreme point which is impossible.
\hfill \qedsymbol{}

There is an analogous version of the Theorem \ref{MAINTHEOREM}, where one can obtain the same result by replacing $c$-extremality with $(c,P)$-extremality. More precisely assume that for each $k\in\mathcal{N}$, the family $P_k=\{Y^{(k)}_l\}_{l=1}^{L_k}$ be a Borel ordered partition of $X_k$ for $k\in\mathcal{N}\backslash\{1\}$. The following theorem states that one can have the uniqueness of the optimal plan of \eqref{TAGMK}. The proof is exactly as the proof of the Theorem \ref{MAINTHEOREM}.

\begin{theorem}\label{MAINTHEOREMCP}
Let  $\{(X_k,\mathcal{B}_{X_k},\mu_k)\}_{k=1}^N$ be a family of Borel probability spaces such that $X_k$'s are Polish spaces, $c:\prod_{k=1}^NX_k\to[0,\infty]$, is lower semi-continuous, $\otimes_{k=1}^N\mu_k$-a.e. finite, and there exists a finite transport plan. Let $P_k=\{Y^{(k)}_l\}_{l=1}^{L_k}$ be an ordered measurable partition of $X_k$, for $k\in\mathcal{N}\backslash\{1\}$. Then 
if optimal plans of \eqref{TAGMK} are $(c,P_N)$-extreme and are such that $\lambda_j=\pi_{\mathcal{P}_j}\#\lambda$ is $(c_j,P_j)$-extreme, for $j\in\mathcal{N}\backslash\{1,N\}$, then the problem \eqref{TAGMK} admits a unique solution.
\end{theorem}

\begin{remark}\label{REM-MAINTH}
It is noteworthy to mention that a combination of $c_j$-extremality and $(c_j,P_j)$-extremality also yields the same results obtained in Theorems \ref{MAINTHEOREM} and \ref{MAINTHEOREMCP}. This is due to the fact that we benefited from the consequence of $c_j$-extremal and $(c_j,P_j)$-extremal measure that is being an extreme point.
\end{remark}

\section{The reduction argument and the gluing principle}\label{SectionDisintegration}
In this section,  we shall construct optimal plans of the $N$-marginal Monge-Kantorovich problem from the associated $k$-marginal problems when $k<N$.
One of the main ingredients in this section is the disintegration theorem which together with the gluing lemma and the reduction argument  allow us to construct optimal plans of the $N$-marginal Monge-Kantorovich problem from the reduced  $k$-marginal ones for  $k<N$. 

\begin{theorem}[Disintegration Theorem]\label{DISINTTH}
    Let $X$ and $Y$ be complete separable metric spaces equipped with Borel probability measures $\mu$ and $\nu$, respectively. For any $\lambda\in\Pi(\mu,\nu)$, there exists a family $(\nu^x)_{x\in X}\subseteq P(Y)$ such that
\begin{enumerate}
    \item the map $x\mapsto \nu^x(B)$ is a Borel measurable map on $X$, for all $B\in \mathcal{B}_Y$,
    \item $\lambda=\nu^x\otimes\mu$, meaning that for all $A\times B\in \mathcal{B}_X\times\mathcal{B}_Y$, we have $\lambda(A\times B)=\int_A\nu^x(B)d\mu(x)$,
    \item the family $(\nu^x)_{x\in X}$, is uniquely determined for $\mu$-a.e. $x\in X$.
\end{enumerate}

\end{theorem}
\noindent
    Proof of existence can be found \cite{DELLACHERIEMEYER} and the uniqueness in \cite{AMBROSIOCAFFARELLIBRENIERBUTTAZZOVILLANI}.

\begin{remark}\label{REM1-APP} Here we are listing some straightforward observations from the latter theorem that will be used in the sequel. 
\begin{enumerate}

    \item\label{REM1-APP-2}
    It is worthy to notice that if there exists a measurable map $T:X\to Y$ such that $\lambda=(\text{id}\times T)\#\mu$ then $\lambda$ has decomposition as $\lambda=\delta_{T(x)}\otimes\mu$ where $\delta_{T(x)}$ stands for the dirac measure supported by the point $T(x).$
\item\label{REM1-APP-3}
Using Theorem \ref{DISINTTH}, one can easily generalize the result to higher dimensions. Assume that $X$, $Y$, and $Z$ are complete separable metric spaces equipped with Borel probability measures $\mu$, $\nu$, and $\gamma$, respectively. 
For  the projection map $\pi_{XY}:X\times Y\times Z\to X\times Y$ and  a measure $\lambda\in\Pi(\mu,\nu,\gamma)$, we denote by $\lambda^{XY}$ the measure $\pi_{XY}\#\lambda$, i.e. the restriction  of $\lambda$ on $X\times Y$. We do the same in the case of $X\times Z$, as well.
Therefore,  for $\lambda\in\Pi(\mu,\nu,\gamma)$, we have
\begin{align}\label{3DIM-DIS-FORM}
    \lambda=\lambda^x\otimes\mu,\quad \text{with}\quad\lambda^x\in\Pi(\nu^x,\gamma^x)\subseteq P(Y\times Z),\quad \text{for}\;\mu\text{-a.e.}\; x\in X,
\end{align}
where $\lambda^{XY}=\nu^x\otimes\mu\in\Pi(\mu,\nu)$ and $\lambda^{XZ}=\gamma^x\otimes\mu \in\Pi(\mu,\gamma)$. Moreover, by applying the disintegration theorem to the measures $\lambda^x$, we have that $\lambda^x=\gamma^{xy}\otimes\nu^x$ with $(\gamma^{xy})_{y\in Y}\subseteq P(Z)$, and consequently, $\lambda=\gamma^{xy}\otimes\nu^x\otimes\mu$, i.e.,
\begin{align*}
    \lambda(A\times B\times C)=\int_A\int_B \gamma^{xy}(C)d\nu^x(y)d\mu(x),\quad \forall A\times B\times C\in \mathcal{B}_X\times \mathcal{B}_Y\times \mathcal{B}_Z.
\end{align*}
\end{enumerate}
\end{remark}

Here is the standard gluing lemma.  We refer to   \cite[Lemma 5.3.2]{AMBROSIO} for a detailed proof. 
\begin{lemma}[Gluing Lemma]\label{GLUING}
   If $\lambda_1=\nu^x\otimes \mu\in \Pi(\mu,\nu)\subseteq P(X\times Y) $ and $\lambda_2=\gamma^x\otimes \mu\in \Pi(\mu,\gamma)\subseteq P(X\times Z) $, then the gluing measure $\lambda =(\nu^x\times \gamma^x)\otimes \mu$ belongs to $\Pi(\mu,\nu, \gamma)\subseteq P(X\times Y\times Z)$, and 
   \begin{align*}
       \pi_{XY}\#\lambda=\lambda^{XY}=\lambda_1,\quad\text{and}\quad \pi_{YZ}\#\lambda=\lambda^{YZ}=\lambda_2.
   \end{align*}
   Moreover, if either $\lambda_1$ or $\lambda_2$ are induced by a transport map then the gluing measure $\lambda$ is unique. 
\end{lemma}

\subsection{The reduction argument and the disintegration theorem}
In this section  we evince an application of the reduction method by means of the disintegration theorem. Assume that
\begin{align*}
 \mathcal{P}\subsetneq\mathcal{N},\quad \mathcal{Q}\subsetneq \mathcal{N}\backslash\mathcal{P},\quad\text{and}\quad   \mathcal{R}=\mathcal{Q}\cup \mathcal{P}.
\end{align*}
Moreover, suppose that $\{\mathcal{Q}_j\}_{j\in\mathbb{J}}$ is a partition of $\mathcal{N}\backslash\mathcal{R}$, for an indexing set $\mathbb{J}$, meaning that
\begin{align*}
\mathcal{Q}_j\neq\emptyset,\quad   \mathcal{Q}_i\cap\mathcal{Q}_j=\emptyset,\quad\text{for}\; i\neq j\in\mathbb{J},\quad \text{and}\quad \cup_{j\in\mathbb{J}}\mathcal{Q}_j=\mathcal{N}\backslash\mathcal{R}.
\end{align*}
\noindent
Before going further, let us clarify about the notations. When we take union of the ordered subsets of $\mathcal{N}$, we consider it as an ordered subset as well. 
For ease of notations, for an ordered subset $\mathcal{P}=\{i_1,\ldots,i_p\}\subseteq\mathcal{N}$, the subset $\Pi(\mu_{i_1},\ldots,\mu_{i_p})$ will be denoted by $\Pi((\mu_k)_{k\in\mathcal{P}})$ and the space $\prod_{k=1}^p X_{i_k}$ will be typified by $\prod_{k\in\mathcal{P}}X_k$. An element in the latter space will be represented by $(x_k)_{k\in\mathcal{P}}$ and it is considered as an ordered $p$-tuple, where $p$ is the number of elements of the index set $\mathcal{P}$. Therefore, action of a function $f:\prod_{k\in\mathcal{P}}X_k\to\mathbb{R}$ on this element will be symbolized as $f((x_k)_{k\in\mathcal{P}})$.\\

\noindent
For each $j \in \mathbb{J}$ let $\mathcal{R}_j=\mathcal{Q}_j\cup \mathcal{P}$ and consider the following optimization problems
\begin{align}
& \inf\bigg\{\int c_{\mathcal{R}}\,d\tau_{\mathcal{R}}\ :\ \tau_{\mathcal{R}}\in\Pi((\mu_k)_{k\in\mathcal{R}})\bigg\}\label{GEN-RED-STAR},\\
   & \inf\bigg\{\int c_{\mathcal{R}_j}\,d\tau_{\mathcal{R}_j}\ :\ \tau_{\mathcal{R}_j}\in\Pi((\mu_k)_{k\in\mathcal{R}_j})\bigg\}\label{GEN-RED-SOME},\quad j\in\mathbb{J},
\end{align}
where $c_{\mathcal{R}}:\prod_{k\in\mathcal{R}}X_k\to\mathbb{R}$ and $c_{\mathcal{R}_j}:\prod_{k\in\mathcal{R}_j}X_k\to\mathbb{R}$ are defined as follows
\begin{align*}
& c_{\mathcal{R}}((x_k)_{k\in\mathcal{R}})=\inf_{\prod_{k\in\mathcal{N}\backslash\mathcal{R}}X_k}\left\{c((x_k)_{k\in\mathcal{N}})-\sum_{k\in\mathcal{N}\backslash\mathcal{R}}\phi_k(x_k)\right\}.\\
 &    c_{\mathcal{R}_j}((x_k)_{k\in\mathcal{R}_j})=\inf_{\prod_{k\in\mathcal{N}\backslash\mathcal{R}_j}X_k}\left\{c((x_k)_{k\in\mathcal{N}})-\sum_{k\in\mathcal{N}\backslash\mathcal{R}_j}\phi_k(x_k)\right\},\quad j\in\mathbb{J}.
\end{align*}
Our goal is to recover optimal plans of \eqref{TAGMK} from optimal plans of \eqref{GEN-RED-STAR} and \eqref{GEN-RED-SOME}. To do so, we may consider a measure $\tau_\mathcal{R}\in\Pi((\mu_k)_{k\in\mathcal{R}})$ as an element of $\Pi(\tau_\mathcal{P},\tau_\mathcal{Q})$ where
\begin{align*}
\tau_\mathcal{P}=\pi_\mathcal{P}\#\tau_\mathcal{R}\quad \text{and}\quad \tau_\mathcal{Q}=\pi_\mathcal{Q}\#\tau_\mathcal{R}.
\end{align*}
Here
\begin{align*}
    \pi_{\mathcal{P}} :\prod_{k\in\mathcal{R}}X_k\to \prod_{k\in\mathcal{P}}X_k,\quad\text{and}\quad  \pi_{\mathcal{Q}} :\prod_{k\in\mathcal{R}}X_k\to \prod_{k\in\mathcal{Q}}X_k,
\end{align*}
are the canonical projection maps. By this point of view, a measure $\tau_\mathcal{R}\in\Pi((\mu_k)_{k\in\mathcal{R}})$ admits the disintegration with respect to the marginal $\tau_\mathcal{P}$ as follows
\begin{align}\label{DIS-TAU-R}
    \tau_\mathcal{R}=\tau^{(x_k)_{k\in\mathcal{P}}}_{\mathcal{Q}}\otimes\tau_{\mathcal{P}}\quad \text{with}\quad\tau^{(x_k)_{k\in\mathcal{P}}}_{\mathcal{Q}}\in P(\prod_{k\in\mathcal{Q}}X_k),\quad\text{for}\ \tau_\mathcal{P}\text{-a.e.}\; (x_k)_{k\in\mathcal{P}}\in \prod_{k\in\mathcal{P}}X_k .
\end{align}
In particular, if $\pi_\mathcal{R}:\prod_{k\in\mathcal{N}}X_k\to \prod_{k\in\mathcal{R}}X_k$ denotes the projection map, then for a measure $\lambda\in \Pi((\mu_k)_{k\in\mathcal{N}})$ the measure $\pi_\mathcal{R}\#\lambda$ can be viewed as an element of $\Pi(\lambda_\mathcal{P},\lambda_\mathcal{Q})$ and  admits a disintegration of the form \eqref{DIS-TAU-R}.

\begin{theorem}\label{GEN-BIGTH}
Assume that for each $j\in\mathbb{J}$, any optimal plan of \eqref{GEN-RED-SOME} is concentrated on the graph of a  measurable map from $\prod_{k\in\mathcal{P}}X_k$  to $ \prod_{k\in\mathcal{Q}_j}X_k$.
Let $\lambda$ be an optimal plan of \eqref{TAGMK} and $\lambda_{\mathcal{R}}=\pi_\mathcal{R}\#\lambda$. Assume that $\lambda_\mathcal{P}=\pi_\mathcal{P}\#\lambda_\mathcal{R}$ and $\lambda_\mathcal{Q}=\pi_\mathcal{Q}\#\lambda_\mathcal{R}$ and that $\lambda_\mathcal{R}$ has the disintegration with respect to the marginal $\lambda_\mathcal{P}$ as follows
\begin{align}\label{DIS-LAMB-R}
    \lambda_\mathcal{R}=\lambda^{(x_k)_{k\in\mathcal{P}}}_{\mathcal{Q}}\otimes\lambda_{\mathcal{P}}\quad \text{with}\quad\lambda^{(x_k)_{k\in\mathcal{P}}}_{\mathcal{Q}}\in P(\prod_{k\in\mathcal{Q}}X_k),\quad\text{for}\ \lambda_\mathcal{P}\text{-a.e.}\; (x_k)_{k\in\mathcal{P}}\in \prod_{k\in\mathcal{P}}X_k.
\end{align}
Then the optimal plan $\lambda$ is of the form
\begin{align}\label{FORMOFOPT}
    \lambda=\big(\prod_{j\in\mathbb{J}}\delta_{T_{\mathcal{Q}_j}((x_k)_{k\in\mathcal{P}})}\times \lambda_{\mathcal{Q}}^{(x_k)_{k\in\mathcal{P}}}\big)\otimes\lambda_{\mathcal{P}}.
\end{align}
for some measurable map $ T_{\mathcal{Q}_j}:\prod_{k\in\mathcal{P}}X_k\to \prod_{k\in\mathcal{Q}_j}X_k  $ and $ j\in\mathbb{J}$.
Moreover, if for each $j\in\mathbb{J}$ the problems \eqref{GEN-RED-STAR} and \eqref{GEN-RED-SOME} have a unique solution, then the optimal plan for \eqref{TAGMK} is also unique.
\begin{proof}
Let $\lambda\in \Pi((\mu_k)_{k\in\mathcal{N}})$ be an optimal plan of \eqref{TAGMK} and $\pi_{\mathcal{R}_j}:\prod_{k\in\mathcal{N}}X_k\to\prod_{k\in\mathcal{R}_j}X_k$ be the canonical projection maps, for $j\in\mathbb{J}$. By Proposition \ref{PROP-3}, each optimal plan $\lambda $ of \eqref{TAGMK} induces optimal plans for \eqref{GEN-RED-STAR} and \eqref{GEN-RED-SOME}. Therefore,  for $j\in\mathbb{J}$, the measures $\lambda_{\mathcal{R}_j}=\pi_{\mathcal{R}_j}\#\lambda$ and $\lambda_{\mathcal{R}}$ are optimal plans of \eqref{GEN-RED-SOME} and \eqref{GEN-RED-STAR} respectively.  On the other hand, through the assumption, the optimal plans of \eqref{GEN-RED-SOME} are concentrated on the graph of some measurable maps from $\prod_{k\in\mathcal{P}}X_k$ to $\prod_{k\in\mathcal{Q}_j}X_k$. Therefore, for $j\in\mathbb{J}$, the measure  $\lambda_{\mathcal{R}_j}$ has the following disintegration
    \begin{align*}
      \lambda_{\mathcal{R}_j}=\delta_{T_{\mathcal{Q}_j}((x_k)_{k\in\mathcal{P}})}\otimes \lambda_{\mathcal{P}}, \quad j\in\mathbb{J}.
    \end{align*}
    for some measurable maps $T_{\mathcal{Q}_j}:\prod_{k\in\mathcal{P}}X_k\to \prod_{k\in\mathcal{Q}_j}X_k$. Consequently, by Lemma \ref{GLUING} the only measure that has marginals $\lambda_{\mathcal{R}_j}$ and $\lambda_{\mathcal{R}}$ has to be $\lambda $ which is given by \eqref{FORMOFOPT}. In detail, to apply this lemma directly, let $\overline{\mathcal{Q}}=\cup_{j\in\mathbb{J}}\mathcal{Q}_j$. Note that by hiding the rearrangement of the product space, we have
    \begin{align*}
  \prod_{k\in\overline{\mathcal{Q}}}X_k= \prod_{j\in\mathbb{J}}\prod_{k\in \mathcal{Q}_j }X_k.
     \end{align*}  
 Let us define the measurable map
    \begin{align*}
        T:\prod_{k\in\mathcal{P}}X_k\to\prod_{k\in\overline{\mathcal{Q}}}X_k,\quad T((x_k)_{k\in\mathcal{P}})=(T_{\mathcal{Q}_j}((x_k)_{k\in\mathcal{P}}))_{j\in\mathbb{J}}.
    \end{align*}
Let 
$\pi_{\overline{\mathcal{Q}}}:\prod_{k\in\mathcal{N}}X_k\to \prod_{k\in\overline{\mathcal{Q}}}X_k $ be the corresponding projection map and put $\lambda_{\overline{\mathcal{Q}}}=\pi_{\overline{\mathcal{Q}}}\#\lambda$.
We consider the measure $\lambda$ as a measure in $\Pi(\lambda_\mathcal{P},\lambda_\mathcal{Q},\lambda_{\overline{\mathcal{Q}}})$ while
\begin{align*}
\lambda_\mathcal{R}\in\Pi(\lambda_\mathcal{P},\lambda_\mathcal{Q}),\quad \text{and}\quad\ (\text{id}\times T)\#\lambda_\mathcal{P}\in\Pi(\lambda_\mathcal{P},\lambda_{\overline{\mathcal{Q}}}).
\end{align*}
Clearly we have $\lambda_\mathcal{R}\in\Pi(\lambda_\mathcal{P},\lambda_\mathcal{Q})$. To see $(\text{id}\times T)\#\lambda_\mathcal{P}\in\Pi(\lambda_\mathcal{P},\lambda_{\overline{\mathcal{Q}}})$ we note that in this case, the disintegrations of $(\text{id}\times T)\#\lambda_\mathcal{P}$ and $\pi_{\mathcal{P}\cup\overline{\mathcal{Q}}}\#\lambda$ with respect to $\lambda_\mathcal{P}$ are the same
\begin{align*}
(\text{id}\times T)\#\lambda_\mathcal{P}=\delta_{T((x_k)_{k\in\mathcal{P}})}\otimes\lambda_{\mathcal{P}}=\prod_{j\in\mathbb{J}}\delta_{T_{\mathcal{Q}_j}((x_k)_{k\in\mathcal{P}})}\otimes\lambda_{\mathcal{P}}=\pi_{\mathcal{P}\cup\overline{\mathcal{Q}}}\#\lambda .
\end{align*}
Then by choosing 
    \begin{align*}
    &(X,\mu)=(\prod_{k\in\mathcal{P}}X_k,\lambda_\mathcal{P}),\quad (Y,\nu)=(\prod_{k\in\mathcal{Q}}X_k,\lambda_\mathcal{Q}),\quad (Z,\gamma)=(\prod_{k\in\overline{\mathcal{Q}}}X_k,\lambda_{\overline{\mathcal{Q}}}),
    \end{align*}
    and
    \begin{align*}
    &\lambda_1=\lambda_\mathcal{R},\quad \lambda_2=(\text{id}\times T)\#\lambda_\mathcal{P},
    \end{align*}
    in Lemma \ref{GLUING}, we have the result. This is due to the fact that via this lemma, the only measure that has marginals $\lambda_1$ and $\lambda_2$ has to be $\lambda$. Consequently, $\lambda$ has to have the disintegration as \eqref{FORMOFOPT}.\\  
For the uniqueness part, let $\tau_{\mathcal{R}}$ and $\tau_{\mathcal{R}_j}$ be the unique solutions to \eqref{GEN-RED-STAR} and \eqref{GEN-RED-SOME}, for $j\in\mathbb{J}$, respectively.
Then each $\lambda\in\Pi((\mu_k)_{k\in\mathcal{N}})$ that is an optimal plan for \eqref{TAGMK} has marginals on $\prod_{k\in\mathcal{R}}X_k$ and $\prod_{k\in\mathcal{R}_j}X_k$ equal to $\tau_{\mathcal{R}}$ and $\tau_{\mathcal{R}_j}$, for $j\in\mathbb{J}$. Hence, thanks to the Lemma \ref{GLUING} and the first part of the proof there is only one measure that has marginals $\tau_\mathcal{R}$ and $\tau_{\mathcal{R}_j}$ which implies the uniqueness of the optimal measures of \eqref{TAGMK}.

\end{proof}
\end{theorem}
Theorem \ref{GEN-BIGTH} is stated in a general form. One of the interesting special cases is when we deal with two-marginal reduced problem. Precisely, in Theorem \ref{GEN-BIGTH}, let 
\begin{align*}
&\mathcal{P}=\{1\},\quad \mathcal{Q}=\{j_0\},\quad \text{for some}\quad  j_0\in\mathcal{N}\backslash\{1\},\\
&\mathbb{J}=\mathcal{N}\backslash\{1,j_0\}\quad\text{and}\quad \mathcal{Q}_j=\{j\},\ \text{for}\  j\in \mathbb{J}.
\end{align*}
Consequently, the problems \eqref{GEN-RED-STAR} and \eqref{GEN-RED-SOME} can be expressed as follows
\begin{align}\label{TWOMARGREDUC}
    \inf\int_{\tau_{1j}\in \Pi(\mu_1,\mu_j)} c_{1j}(x_1,x_j)\,d\tau_{1j}(x_1,x_j),\quad j\in\mathcal{N}\backslash\{1\},
\end{align}
where the function  $c_{1j}:X_1\times X_j\to\mathbb{R}$ is defined by
\begin{align*}
&c_{1j}(x_1,x_j)=\inf_{\prod_{k=2,\neq j}^N X_k}\big\{c(x_1,\ldots,c_N)-\sum_{k=2,\neq j}^N\phi_k(x_k)\big\},\quad j\in\mathcal{N}\backslash\{1\}.
\end{align*}
In the following theorem for $j\in\mathcal{N}\backslash\{1\}$ we denote by $\pi_{1j}$ the projection map from $\prod_{k=1}^NX_k$ onto $X_1\times X_j$. Recall that the restriction of an element $\lambda\in\Pi(\mu_1,\ldots,\mu_N)$ on $X_1\times X_j$ that is $\pi_{1j}\#\lambda\in\Pi(\mu_1,\mu_j)$ has the decomposition with respect to the first marginal as follows
 \begin{align*}
     \pi_{1j}\#\lambda=\mu_j^{x_1}\otimes\mu_1,\quad \mu_j^{x_1}\in P(X_j),\quad j\in\mathcal{N}\backslash\{1\}.
 \end{align*}
Now, we state a more applicable version of Theorem \ref{GEN-BIGTH}.

\begin{theorem}\label{BIGTH} 
For each $j\in\mathcal{N}\backslash\{1,j_0 \}$, let the optimal plan of \eqref{TWOMARGREDUC} be induced by the unique transport map $T_1: X_1\to X_j$. Assume that $\lambda$ is an optimal plan of \eqref{TAGMK} and is such that $\pi_{1,j_0 }\#\lambda=\mu_{j_0}^{x_1}\otimes \mu_1$. Then $\lambda$ is of the form
\begin{align*}
   \big(\prod_{j=2}^{j_0-1}\delta_{T_j(x_1)}\times \mu^{x_1}_{j_0}\times\prod_{j=j_0+1}^{N}\delta_{T_j(x_1)}\big) \otimes \mu_1.
\end{align*}
If for  $j=j_0$, the problem \eqref{TWOMARGREDUC} admits a unique solution, then \eqref{TAGMK} has a unique solution as well.

\end{theorem}
\begin{remark}
    Theorem \ref{BIGTH} determines the necessary and sufficient conditions for the existence of the solutions for the multi-marginal Monge problem. In fact, \eqref{TAGM} has a unique solution if and only if the two-marginal problem \eqref{TWOMARGREDUC} admits a unique solution induced by a transport map, for all $j\in\mathcal{N}\backslash\{1\}$.
\end{remark}
\subsection{Optimal plans concentrating on two maps}
Consider three probability spaces $(X,\mathcal{B}_X,\mu)$, $(Y,\mathcal{B}_Y,\nu)$ and $(Z,\mathcal{B}_Z,\gamma)$. We aim at finding the general form of the solutions to the optimization problem associated with the cost function  $c:X\times Y\times Z\to\mathbb{R}$ via connecting the optimal plans of optimization problem obtained by the  reduction argument on the spaces $X\times Y$ and $X\times Z$. Consider the following optimization problem
\begin{align}
     & \inf_{\lambda\in\Pi(\mu,\nu,\gamma)}\int c(x,y,z)\,d\lambda.\label{OPT-CHAR-THREE}
\end{align}
Assume that the dual problem to \eqref{OPT-CHAR-THREE} has the following solution
\begin{align*}
    (\phi_1,\phi_2,\phi_3)\in L_1(X,\mu)\times L_1(Y,\nu)\times L_1(Z,\gamma).
\end{align*}
Put
\begin{align}
    &c_1:X\times Y\to\mathbb{R},\quad c_1(x,y)=\inf_{z\in Z}\big\{c(x,y,z)-\phi_3(z)\big\},\label{COSTC1}\\
    &c_2:X\times Z\to\mathbb{R},\quad c_2(x,z)=\inf_{y\in Y}\big\{c(x,y,z)-\phi_2(y)\big\},\label{COSTC2}
\end{align}
The corresponding two-marginal Monge-Kantorovich problems associated with cost functions \eqref{COSTC1} and \eqref{COSTC2} are as follows
\begin{align}
      & \inf_{\lambda_1\in\Pi(\mu,\nu)}\int c_1(x,y )\,d\lambda_1,\label{OPT-CHAR-TWO-ONE}\\
         & \inf_{\lambda_2\in\Pi(\mu,\gamma)}\int c_2(x ,z)\,d\lambda_2.\label{OPT-CHAR-TWO-TWO}
\end{align}
Let $\pi_{XY}$ and $\pi_{XZ}$ denote the natural projections from $X\times Y\times Z$ onto $ X\times Y$ and $ X\times Z$, respectively. In the following theorem, we establish the general form of optimal plans of \eqref{OPT-CHAR-THREE} when the problems \eqref{OPT-CHAR-TWO-ONE} and \eqref{OPT-CHAR-TWO-TWO} have optimal plans concentrating on multiple measurable maps.
\begin{theorem}\label{THEOREM-2-TWIST-GEN}
The following assertions hold;
\begin{enumerate}
    \item
    Let  $\lambda$  be an optimal plan  \eqref{OPT-CHAR-THREE} such that the restrictions of $\lambda$  on $X\times Y$ and $X\times Z$ are of the form 
    \begin{align}
    &\pi_{XY}\#\lambda=\big(\alpha(x)\delta_{T_1(x)}+(1-\alpha(x))\delta_{T_2(x)}\big)\otimes\mu,\quad \alpha:X\to [0,1],\label{DIS-NU}\\
&\pi_{XZ}\#\lambda=\big(\beta(x)\delta_{G_1(x)}+(1-\beta(x))\delta_{G_2(x)}\big)\otimes\mu,\quad \beta:X\to [0,1],\label{DIS-GAMMA}
\end{align} 
  where  $ T_1, T_2:X\to Y$,   $G_1, G_2:X\to Z$ and the scalar functions $\alpha,\beta$ are measurable. Then $\lambda $ is of the form   
\begin{align}\label{DIS-LAM}
    \lambda=\big(\sum_{i,j=1}^2L_{ij}(x)\delta_{T_i(x)}\times\delta_{G_j(x)}\big)\otimes\mu,
\end{align}
where the maps $L_{ij}:X\to [0,1]$ are measurable maps. Moreover, the $4$-tuples $(L_{11},L_{12},L_{21},L_{22})$ solves the following system of equations
\begin{align}\label{TAGA}
\begin{array}{ll}
  L_{11}(x)+L_{12}(x)=\alpha(x),&\\
  L_{21}(x)+L_{22}(x)=1-\alpha(x),&\\
  L_{11}(x)+L_{21}(x)=\beta(x),&\\
  L_{12}(x)+L_{22}(x)=1-\beta(x),&
\end{array}
\end{align}    
for $\mu$-a.e. $x\in X$.
\item
If problems \eqref{OPT-CHAR-TWO-ONE} and \eqref{OPT-CHAR-TWO-TWO} admit unique solutions of the form \eqref{DIS-NU} and \eqref{DIS-GAMMA} then the problem \eqref{OPT-CHAR-THREE} has two optimal plans which are extreme points of $\Pi(\mu,\nu,\gamma)$, and any other plan is a convex combination of these two extreme points.  Particularly, the problem \eqref{OPT-CHAR-THREE} has a unique optimal plan if and only if it is associated with the $4$-tuples $(L^*_{11},L^*_{12},L^*_{21},L^*_{22})$ which satisfies the following relations
\begin{align}\label{TAGB}
\begin{array}{ll}
    &L^*_{11}(x)=\alpha(x)\beta(x),\\
    &L^*_{12}(x)=\alpha(x)(1-\beta(x)),\\
    &L^*_{21}(x)=(1-\alpha(x))\beta(x),\\
    &L^*_{22}(x)=(1-\alpha(x))(1-\beta(x)).
    \end{array}
\end{align}
\end{enumerate}
\begin{proof}
\begin{enumerate}
    \item
    First, we note that by Proposition \ref{PROP-3}, for an optimal plan $\lambda$ of \eqref{OPT-CHAR-THREE}, the marginals $\pi_{XY}\#\lambda$ and $\pi_{XZ}\#\lambda$ are optimal plans of \eqref{OPT-CHAR-TWO-ONE} and \eqref{OPT-CHAR-TWO-TWO}, respectively. Thus, by assumption, they are of the forms \eqref{DIS-NU} and \eqref{DIS-GAMMA}. Generally, the marginals of $\lambda\in\Pi(\mu,\nu,\gamma)$ with on $X\times Y$ and $X\times Z$ have the following disintegrations
    \begin{align*}
     \pi_{XY}\#\lambda=\nu^x\otimes\mu\quad \text{and} \quad \pi_{XZ}\#\lambda=\gamma^x\otimes\mu,\quad \nu^x\in P(Y),\ \gamma^x\in P(Z),
    \end{align*}
   Moreover, we have 
\begin{align*}
\lambda=\lambda^x\otimes\mu,\quad \text{where}\quad \lambda^x\in\Pi(\nu^x,\gamma^x),\quad\mu\text{-a.e.}\  x\in X.
\end{align*}
    In this problem, we have
\begin{align}\label{DIS-MARG}
    \nu^x=\sum_{i=1}^2\alpha_i(x)\delta_{T_i(x)},\quad \text{and}\quad \gamma^x=\sum_{i=1}^2\beta_i(x)\delta_{G_i(x)}.
\end{align}
Thus, for a fixed $x\in X$, we have
\begin{align*}
    \text{Spt}(\lambda^x)\subseteq\big\{(T_i(x),G_j(x))\ :\ i,j=1,2\big\}.
\end{align*}
    which implies that
    \begin{align*}
        \lambda^x=\sum_{i,j=1}^2L_{ij}(x)\delta_{T_i(x)}\times\delta_{G_j(x)}
    \end{align*}
    Thus, \eqref{DIS-LAM} holds. Moreover, from the fact that $\lambda^x$ has marginals $\nu^x$ and $\gamma^x$ as they are in \eqref{DIS-MARG}, one can obtain the system of equations in \eqref{TAGA}.
\item
In this part, due to the fact that the optimization problems \eqref{OPT-CHAR-TWO-ONE} and \eqref{OPT-CHAR-TWO-TWO} admit unique solutions as they are in \eqref{DIS-NU} and \eqref{DIS-GAMMA}, then the infimum in \eqref{OPT-CHAR-THREE} can be  taken  over $\Pi(\lambda^{XY},\gamma)\cap\Pi(\lambda^{XZ},\nu)$ instead.
Let us find the solutions of the system of equations \eqref{TAGA}. We drop $x$ for ease of notation whenever it is needed. We have,
\begin{align}
   & L_{12}=\alpha-L_{11},\label{r11}\\
    &L_{21}=\beta-L_{11},\label{r22}\\
    & L_{22}=1-(\alpha+\beta)+L_{11}\label{r33},
\end{align}
Using \eqref{r11}-\eqref{r33} and the fact that $0\leq L_{ij}\leq 1$,    if $(L_{11}, L_{12}, L_{21}, L_{22})$ solves the system \eqref{TAGA} then $(L_{11}, L_{12}, L_{21}, L_{22}) \in K_{\alpha,\beta}$ where 
\begin{align*}
    K_{\alpha,\beta}:=\bigg\{(L_{11},\alpha-L_{11},\beta-L_{11},1-(\alpha+\beta)+L_{11})\ : \ \max{(0,\alpha+\beta -1)}\leq L_{11}\leq \min{(\alpha,\beta)} \bigg\},
\end{align*}
As one can see, if $L_{11}$ is an end-point of the interval 
\begin{align*}
 [\max{(0,\alpha+\beta -1)},\min{(\alpha,\beta)}],
\end{align*}
then the $4$-tuples 
\begin{align}\label{4TUPLE}
(L_{11},\alpha-L_{11},\beta-L_{11},1-(\alpha+\beta)+L_{11}),
\end{align}
is an extreme point of the set $K_{\alpha,\beta}$. So, we claim that the following measurable maps give us two extreme points of $\Pi(\mu,\nu,\gamma)=\Pi(\lambda^{XY},\gamma)\cap\Pi(\lambda^{XZ},\nu)$
\begin{align*}
    &\overline{L}_{11}:X\to [0,1],\quad \overline{L}_{11}(x)= \max{(0,\alpha(x)+\beta(x) -1)},\\
    &\overline{\overline{L}}_{11}:X\to [0,1],\quad \overline{\overline{L}}_{11}(x)= \min{(\alpha(x),\beta(x))}.
\end{align*}
By \eqref{4TUPLE}, let
\begin{align*}
(\overline{L}_{11},\overline{L}_{12},\overline{L}_{21},\overline{L}_{22})\quad\text{and}\quad (\overline{\overline{L}}_{11},\overline{\overline{L}}_{12},\overline{\overline{L}}_{21},\overline{\overline{L}}_{22})
\end{align*}
denote the two $4$-tuples associated with $\overline{L}_{11}$ and $\overline{\overline{L}}_{11}$, respectively. Moreover, let $\overline{\lambda}$ and $\overline{\overline{\lambda}}$ be induced by these two $4$-tuples as in \eqref{DIS-LAM}. Since, for $\mu$-almost every $x\in X$, these two $4$-tuples are extreme points of $K_{\alpha,\beta}$, then it can be seen that
\begin{align*}
   & \overline{\lambda}^x=\sum_{i,j=1}^2\overline{L}_{ij}(x)\delta_{T_i(x)}\delta_{G_j(x)},\quad\text{and}\quad\overline{\overline{\lambda}}^x=\sum_{i,j=1}^2\overline{\overline{L}}_{ij}(x)\delta_{T_i(x)}\delta_{G_j(x)},
\end{align*}
are extreme points of
\begin{align*}
\Pi\Big(\alpha(x)\delta_{T_1(x)}+(1-\alpha(x))\delta_{T_2(x)},\beta(x)\delta_{G_1(x)}+(1-\beta(x))\delta_{G_2(x)}\Big).
\end{align*}
Hence, via Lemma \ref{LEM2}, the measures $\overline{\lambda}$ and $\overline{\overline{\lambda}}$ are extreme points of
\begin{align*}
\Pi(\lambda^{XY},\gamma)\cap\Pi(\lambda^{XZ},\nu).
\end{align*}
Note that these two extreme points generate all optimal plans of \eqref{OPT-CHAR-THREE}. Indeed, if $\lambda$ is an optimal plan of \eqref{OPT-CHAR-THREE} then, by the first part, we know that it is of the form \eqref{DIS-LAM}. Let $(L_{11},L_{12},L_{21},L_{22})$ be the $4$-tuples of measurable maps associated with $\lambda$. Then we have
\begin{align*}
    \overline{L}_{11}(x)\leq L_{11}(x)  \leq  \overline{\overline{L}}_{11}(x),\quad \mu-\text{a.e.}\; x\in X.
    \end{align*}
Hence, one can find a unique map $\theta :X\to [0,1]$ so that
\begin{align*}
    L_{11}(x)=\theta(x)\overline{L}_{11}(x)+(1-\theta(x))\overline{\overline{L}}_{11}(x),\quad \mu-\text{a.e.}\; x\in X,
\end{align*}
and consequently, all other maps $L_{12}$, $L_{21}$, and $L_{22}$ can be determined.\\
To have a unique extreme point, the necessary and sufficient condition is that the two extreme points $\overline{L}_{11}$ and $\overline{\overline{L}}_{11}$ of $K_{\alpha,\beta}$ coincide. The following computations reveal that the unique extreme point satisfies \eqref{TAGB}.
\begin{align*}
    \max{(0,\alpha+\beta -1)}=\min{(\alpha,\beta)}&\Leftrightarrow \left\{
    \begin{array}{ll}
        0=\min{(\alpha,\beta)} & \Leftrightarrow \alpha=0\;\text{or}\;\beta=0\\
         \text{or}& \text{or}\\
      \alpha+\beta -1 =\min{(\alpha,\beta)}  &\Leftrightarrow \beta=1\;\text{or}\;\alpha=1
    \end{array}
    \right.
\end{align*}
Let us denotes the unique extreme point by $L^*_{11}$. To show \eqref{TAGB}, we just deal with one case, the other cases follow similarly. Let $\alpha=0$. Then we have
\begin{align*}
    &L^*_{11}=0=\alpha \beta,\\
    &L^*_{12}=0=\alpha (1-\beta),\\
    &L^*_{21}=\beta=(1-\alpha)\beta,\\
    &L^*_{22}=1-\beta=(1-\alpha)(1-\beta).
\end{align*}
So, we are done.

\end{enumerate}
\end{proof}

\end{theorem}

\begin{remark}
One of the applications of Theorem \ref{THEOREM-2-TWIST-GEN} can be seen when the conditions of Theorem \ref{MOMEN-CHAR-TH} are satisfied for the problems \eqref{OPT-CHAR-TWO-ONE} and \eqref{OPT-CHAR-TWO-TWO}. Indeed, let $X$ be a complete separable Riemannian manifold and that $Y$ and $Z$ are Polish spaces. If
\begin{enumerate}
    \item[(i)]
    the function $c$ is bounded continuous such that any $c$-concave function is $\mu$-almost everywhere differentiable on its domain,
\item[(ii)]
the functions $c_i$ given in \eqref{COSTC1} and \eqref{COSTC2} satisfy $2$-twist condition, 
\item[(iii)]
the measure $\mu$ is non-atomic,
\end{enumerate}
then the optimal plans of \eqref{OPT-CHAR-TWO-ONE} and \eqref{OPT-CHAR-TWO-TWO} are of the forms \eqref{DIS-NU} and \eqref{DIS-GAMMA}.
Moreover, this theorem can be used  to describe the  optimal plan  of a  three-dimensional Monge-Kantorovich problem which has a unique solution but  the optimal plan is not necessarily concentrated on a single map.
\end{remark}

\section{Applications}\label{SectionApplication}
In this section we provide several applications of the reduction argument from  Section 3. In the first application, we address a $3$-marginal problem with a quadratic cost  for which the unique optimal plan  may concentrate on the graphs of several measurable maps as opposed to just one map.  In our second application, we consider a cost function which is a mix of the quadratic  and the Euclidean norm.  In our last application,  we shall provide a short proof for the well-known  Gangbo and \'Swi\c ech quadratic cost based the reduction argument.

\subsection{ A three dimensional marginal problem with a  quadratic cost function}
Consider the following cost function
\begin{align*}
   (x,y,z) \to |x-y|^2+|x-z|^2+|y-z|^2,
\end{align*}
with the domain $X\times Y\times Z \subset \mathbb{R}^n \times \mathbb{R}^n \times \mathbb{R}^n$, where $|\cdot |$ is the standard Euclidean norm on $\mathbb{R}^n$ and $Y=Y_1\cup Y_2$. Here, $X$, $Y_1$ and $Y_2$  are parallel planes in  $\mathbb{R}^{n}$ of the following forms
\begin{align*}
    & X=\{x\in\mathbb{R}^n\ : \langle\mathrm{n}_0,x\rangle=d_0\},\quad Y_1=\{y\in\mathbb{R}^n\ : \langle\mathrm{n}_0,y\rangle=d_1\}\quad\text{and}\quad Y_2=\{y\in\mathbb{R}^n\ : \langle\mathrm{n}_0,y\rangle=d_2\},
\end{align*}
where $\mathrm{n}_0$ denotes the normal vector of the parallel planes and we have set $Y=Y_1\cup Y_2$. Moreover, we assume that $Z=\cup_{l=1}^L\mathcal{S}_l$, in which $\mathcal{S}_l$'s are the boundaries of the nested strictly convex bodies $\Omega_l\subseteq\mathbb{R}^n$ with $\Omega_1 \subset \Omega_2\subset ...\subset \Omega_L$.

\begin{figure}[h]
    \centering
    \includegraphics[scale=.5]{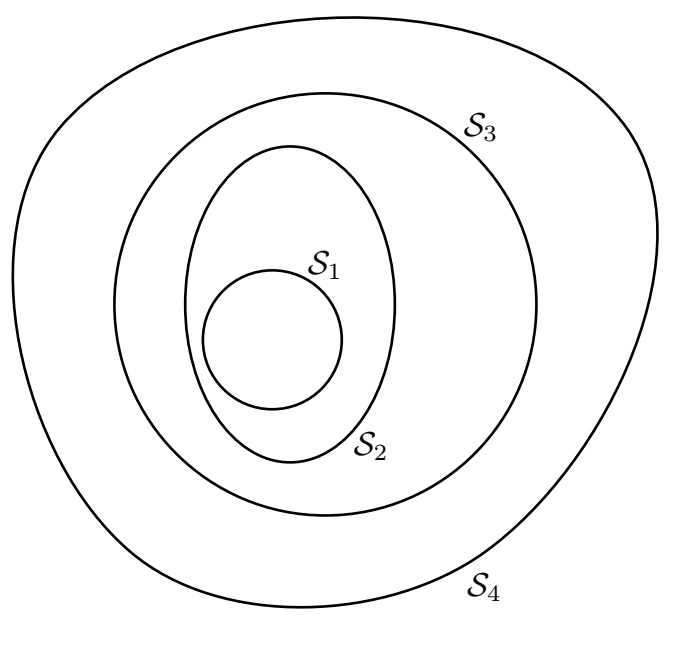}
    \caption{Nested convex sets}
    \label{fig:my_label}
\end{figure}
Before stating the general result, we give an example which distinguishes our result from Gangbo-\'Swi\c ech theorem. In fact, the following example illustrates the fact that under our assumptions, the unique solution is not concentrated on the graph of a  single map.\\
\begin{example}\label{EX-NON-UNIQ}
Consider the $3$-dimensional spaces $X$, $Y$ and $Z$ as follows:
\begin{align*}
    &X=Y=\big\{ x=(x_1,x_2,x_3)\in\mathbb{R}^3\ :\ x_3=0 \big\},\\
    &Z=\big\{z=(z_1,z_2,z_3)\in\mathbb{R}^3\ :\ |z|^2=z_1^2+z_2^2+z_3^2=1\big\}=S^1
\end{align*}
Equip $X$ and $Y$ with a compactly supported  probability measure $\mu$ which is absolutely continuous with respect to the Lebesgue measure on $\mathbb{R}^2$, i.e.,
\begin{align*}
    d\mu(x_1,x_2,0)=h(x_1,x_2)d\mathcal{L}^2(x_1,x_2),
\end{align*}
where $d\mathcal{L}^2$ denotes Lebesgue measure on $\mathbb{R}^2$, and $h$ is a non-negative integrable function. Moreover, let us impose on $Z$  the probability measure
\begin{align*}
d\gamma=\alpha(z_1,z_2,z_3) dS,
\end{align*}
where $dS$ denotes the surface area element and,  $\alpha$ is a positive continuous function satisfying  the following property
\begin{align}\label{EXEQ2}
    \alpha(z_1,z_2,z_3)=\alpha(z_1,z_2,-z_3).
\end{align}
For instance, $\alpha$ can be chosen to be the constant $\frac{1}{4\pi}$. By the property \eqref{EXEQ2} we have that
\begin{align}\label{EXEQ3}
   \int_{Z}f(z_1,z_2,z_3) d\gamma(z_1,z_2,z_3)=\int_{Z}f(z_1,z_2,-z_3)d\gamma(z_1,z_2,z_3),\quad\forall f\in C_b(Z).
\end{align}
Now let us consider the following three-marginal problem
\begin{align}\label{EXEQ1}
    \inf\big\{ \int_{X\times Y\times Z}\big(|x-y|^2+|x-z|^2+|y-z|^2\big)\, d\lambda(x,y,z)\ :\ \lambda\in\Pi(\mu,\mu,\gamma) \big\}
\end{align}
We shall show that if the problem \eqref{EXEQ1} has a solution induced by a transport map, then it fails to have a unique solution. We remark that since $X$ and $Y$ are  equipped with the same probability measure $\mu,$ for any optimal plan $\lambda\in\Pi(\mu,\mu,\gamma)$ we have that
\begin{align}\label{EXEQ4}
   (x,y,z)\in\text{Spt}(\lambda)\Longleftrightarrow (y,x,z)\in\text{Spt}(\lambda).
\end{align}
To see \eqref{EXEQ4}, we just need to note that any ball centered at $ (y,x,z) $ induces a ball with the same radius centered at $ (x,y,z) $.
On the other hand, if $\lambda$ is an optimal plan, then by $c$-cyclically monotonicity of $\lambda$ and \eqref{EXEQ4}, for any $(x,y,z)\in\text{Spt}(\lambda)$, we have that
\begin{align*}
    c(x,y,z)+c(y,x,z)\leq c(x,x,z)+c(y,y,z),
\end{align*}
from which we get
\begin{align*}
    |x-y|\leq 0.
\end{align*}
Hence, 
\begin{align}\label{EXEQ5}
    x=y,\quad \forall (x,y,z)\in\text{Spt}(\lambda).
\end{align}
Let \eqref{EXEQ1} admits a solution induced by a transport map, that is, an optimal transport of the form $\lambda=(\text{id}\times T\times G)\#\mu$, where
\begin{align*}
  T:X\to Y,\quad \text{and} \quad  G:X\to Z,
\end{align*}
are measurable maps. Note that $G=(G_1,G_2,G_3)$ where $G_1, G_2, G_3: X \to \mathbb{R}$ are scalar functions. First, we note that by \eqref{EXEQ5}, we have $T=\text{id}$. Thus,
\begin{align*}
    \lambda=(\text{id}\times \text{id}\times G)\#\mu.
\end{align*}
Now, we shall show that there is another optimal plan $\tilde{\lambda}$ which is different from $\lambda$. Define
\begin{align*}
    \tilde{G}:X\to Z,\quad \tilde{G}=(G_1,G_2,-G_3),
\end{align*}
and correspondingly define
\begin{align*}
     \tilde{\lambda}=(\text{id}\times \text{id}\times \tilde{G})\#\mu.
\end{align*}

We claim that 
\begin{itemize}
    \item[(i)]
$\tilde{G}\#\mu=\gamma$,
    \item[(ii)]
    $\tilde{\lambda}$ is an optimal plan.
\end{itemize}
To see $(\text{i})$, let $f\in C_b(Z)$. Then by using \eqref{EXEQ3}, $G\#\mu=\gamma$, and the fact that the function $F(z_1,z_2,z_3):=f(z_1,z_2,-z_3)$ belongs to $C_b(Z)$, we have
\begin{align*}
    \int_Z f(z)\,d\tilde{G}\#\mu(z)&=\int_X f(\tilde{G}(x))\,d\mu(x)\\
    &=\int_X F(G(x))\,d\mu(x)\\
    &=\int_Z F(z_1,z_2,z_3)\,d\gamma (z_1,z_2,z_3)\\
    &=\int_Z f(z_1,z_2,-z_3)\,d\gamma (z_1,z_2,z_3)\\
    &=\int_Z f(z)\,d\gamma (z)
\end{align*}
which proves $(\text{i})$. For the part $(\text{ii})$, we take advantage of \eqref{EXEQ5}. Indeed, by knowing the support of the optimal plans of \eqref{EXEQ1}, the claim can be established through the following computations
\begin{align*}
    \int c(x,y,z)\,d\lambda&=\int_X |x-x|^2+2|x-G(x)|^2\,d\mu(x)\\
    &=2\int_X |x-G(x)|^2\,d\mu(x),\quad x=(x_1,x_2,0)\\
    &=2\int_X \Big(|x_1-G_1(x_1,x_2,0)|^2+|x_2-G_2(x_1,x_2,0)|^2+|0-G_3(x_1,x_2,0)|^2\Big)\,d\mu(x_1,x_2,0)\\
       &=2\int_X |x-\tilde{G}(x)|^2\,d\mu(x)\\
       &=\int c(x,y,z)\,d\tilde{\lambda}.
\end{align*}
Hence, $\tilde{\lambda}$ is an optimal plan.
\hfill \qedsymbol{}
\end{example}
Consider the following equivalent problem  in terms of the inner product in $\mathbb{R}^n,$
\begin{align}\label{3MARG-QUAD}
    \sup_{\lambda\in\Pi(\mu,\nu,\gamma)}\int_{X\times Y\times Z} (\langle x,y\rangle+\langle x,z\rangle+\langle y,z\rangle)\,d\lambda .
\end{align}
\begin{figure}[h]
    \centering
    \includegraphics[scale=.2]{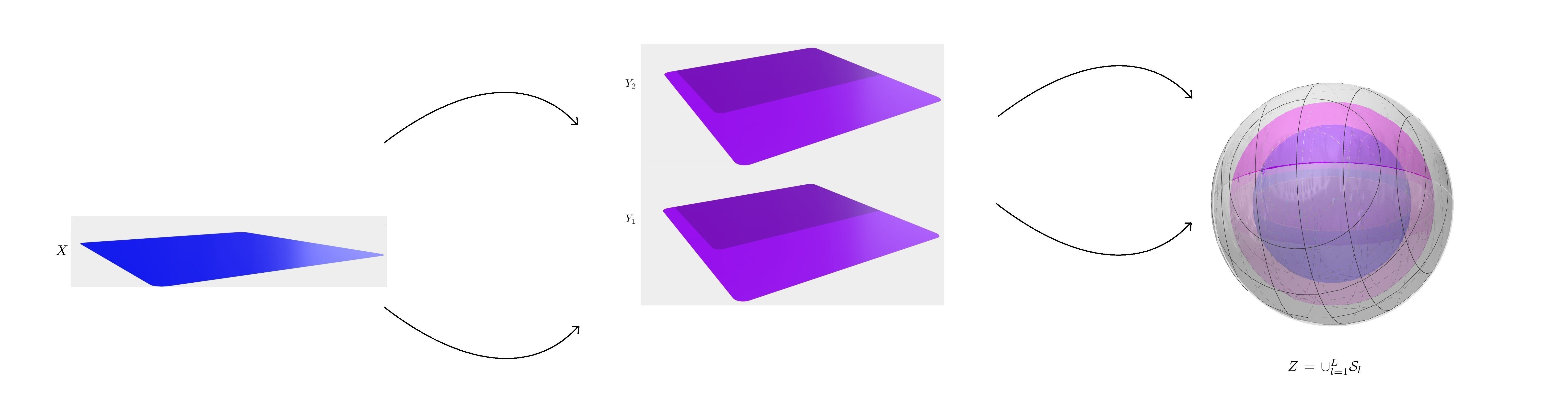}
    \caption{An illustration of the domains for the problem \eqref{3MARG-QUAD}  with  $Z=\cup_{l=1}^L\mathcal{S}_l$ is sketched for $L=3$.}
    \label{fig:my_label}
\end{figure}

We are now ready to state our first application.

\begin{theorem} \label{3MARG-QUADTHEOREM}
Let $X$,$Y_1$ and $Y_2$  be parallel planes in $\mathbb{R}^n$ equipped with compactly supported Borel probability measures $\mu$, $\nu_1$ and $\nu_2$, respectively. Assume that $Z=\cup_{l=1}^L\mathcal{S}_l$ is such that $\mathcal{S}_l$'s are boundary of bounded nested bodies $\Omega_l$ in $\mathbb{R}^n$, i.e., $\Omega_{l}\subseteq \Omega_{l+1}$, together with the Borel probability measure $\gamma$. Assume that $\mu$ and $\nu_2$ are absolutely continuous with respect to the $n-1$ dimensional Lebesgue measure on $X$ and $Y_2$, and consider $Y=Y_1\cup Y_2$ equipped with the Borel probability measure $\nu=t\nu_1+(1-t)\nu_2$, for some $t\in (0,1)$.  If $\gamma$ is  absolutely continuous with respect to the Lebesgue measure on $\mathbb{R}^{n-1}$  on each coordinate charts on  $Z$, then the problem \eqref{3MARG-QUAD} admits a unique solution which is not necessarily concentrated on the graph of an optimal transport map.
\end{theorem}
We shall need some preliminaries before proving this Theorem. We note that since the marginals are compactly supported and the cost function is continuous then the problem \eqref{3MARG-QUAD} and its dual admit solutions. Let 
\begin{align*}
(\phi_1,\phi_2,\phi_3)\in L_1(X,\mu)\times L_1(Y,\nu)\times L_1(Z,\gamma),
\end{align*}
be a solution to the Kantorovich dual problem to \eqref{3MARG-QUAD} which can be considered to be convex. By sticking to  the notations in \eqref{GEN-TRI}, we define the function $c_1:X\times Y\to \mathbb{R}$ as
\begin{align*}
 c_1(x,y)=\sup_{z\in Z}\big\{c(x,y,z)-\phi_3(z)\big\},
\end{align*}
where \[c(x,y,z)=\langle x,y\rangle+\langle x,z\rangle+\langle y,z\rangle.\]
It follows that 
\begin{align}\label{c1-3MARGQUAD}
 c_1(x,y)=\langle x,y\rangle+\psi(x+y) ,\quad\text{with}\quad    \psi(t)=\sup_{z\in Z}\big\{\langle t,z\rangle-\phi_3(z)\big\}.
\end{align}
Additionally, we shall need to consider the reduced version of the problem \eqref{3MARG-QUAD} corresponding to $c_1$,
\begin{align}\label{c1-REDUC-PROB}
    \sup_{\tau\in\Pi(\mu,\nu)}\int (\langle x,y\rangle+\psi(x+y))\,d\tau .
\end{align}
We need the following lemma in the proof of Theorem \eqref{3MARG-QUADTHEOREM}. Here, for an element $z^*\in Z$, the unit outward normal to $Z$ at $z^*$ will be  denoted by $\mathrm{n}(z^*)$.

\begin{lemma}\label{ext-Gen} Let $\lambda \in \Pi(\mu,\nu,\gamma)$ be an optimal plan of \eqref{3MARG-QUAD}. Then there exists a full $\gamma$-measure subset $Z_0$ of $Z$ such that for all $z^* \in Z_0$ if there exist $(x,y)$ and $(x',y')$ in $X\times Y$ such that   $(x,y,z^*)$ and $(x',y',z^*)$ belong to $\text{Spt}(\lambda)$ then there exists $\alpha \in \mathbb{R}$ such that
\begin{align*}
    \mathrm{n}(z^*)=\alpha (y'-y+x'-x).
\end{align*}

\begin{proof}
The proof follows from the  convexity of the potential $\phi_3$ and the absolute continuity of $\gamma$. In fact, since $\phi_3$ can be expressed as
\begin{align*}
    \phi_3(z)=\sup_{X\times Y}\big\{\langle z,x+y\rangle +\langle x,y\rangle-(\phi_1(x)+\phi_2(y))\big\},
\end{align*}
together with the fact that $\phi_3\in L_1(Z,\gamma)$ it can be obtained that $\phi_3$ is convex, $\gamma$-almost everywhere finite and differentiable on its effective domain. In fact, since $\phi_3$ is locally Lipschitz on $\cup_{l=1}^L \Omega_l$, then its restriction on $Z=\cup_{l=1}^L \mathcal{S}_l$ is  locally Lipschitz with respect to the induced distance on $Z$. Therefore, via Rademacher's theorem, the function $\phi_3|_Z$ and the manifolds $\mathcal{S}_l$ are differentiable for $\mathcal{L}^{n-1}$-a.e. points in $Z$. Here $\mathcal{L}^{n-1}$ stands for the Lebesgue measure on $\mathbb{R}^{n-1}$. Let  $Z_0=\text{Dom}(D\phi_3|_Z)$. Then by absolute continuity of $\gamma$ we have $\gamma(Z_0)=1$. Let $z^*\in Z_0$ and $(x,y),(x',y')\in X\times Y$ be such that
\begin{align*}
    (x,y,z^*),\; (x',y',z^*)\in\text{Spt}(\lambda).
\end{align*}
Then, because of the differentiability of $\phi_3$ at $z^*$ we have
\begin{align}\label{ZZ-DIF}
D_z c(x,y,z^*)=D_z c(x',y',z^*).
\end{align}
Now, let $f:U_{z^*}\to\mathbb{R}^{n-1}$ be any coordinate chart for $Z$ where $U_{z^*}$ is an open neighborhood of $z^*$. From \eqref{ZZ-DIF} it follows that
\begin{align*}
    \big(x+y-(x'+y')\big)\cdot Df(z^*)=0,
\end{align*}
from which the result follows.
\end{proof}

\end{lemma}
As a result of $c$-cyclical monotonicity of $\text{Spt}(\lambda)$, we have the following applicable lemma.
\begin{lemma}\label{LEMCYC}
    If $\lambda\in\Pi(\mu,\nu,\gamma)$ is an optimal plan for the maximization problem associated with $c(x,y,z)=\langle x,y\rangle+\langle x,z\rangle+\langle y,z\rangle$, and $(x_1,y_1,z_1),\, (x_2,y_2,z_2)\in\text{Spt}(\lambda)$, then we have
    \begin{align}
        \langle x_1-x_2+y_1-y_2,z_1-z_2\rangle\geq 0.\label{MONOTON-INEQQQ}
    \end{align}
    \begin{proof}
        In fact, inequality \eqref{MONOTON-INEQQQ} can be stated as the following ones as well,
        \begin{align*}
                    \langle y_1-y_2+z_1-z_2,x_1-x_2\rangle\geq 0,
        \end{align*}
        or
\begin{align*}
                    \langle x_1-x_2+z_1-z_2,y_1-y_2\rangle\geq 0.
        \end{align*}
        To see \eqref{MONOTON-INEQQQ}, we simply use $c$-cyclical monotonicity of  $\text{Spt}(\lambda)$, that is,
        \begin{align*}
            c(x_1,y_1,z_1)+c(x_2,y_2,z_2)\geq c(x_1,y_1,z_2)+c(x_1,y_1,z_2),\quad\forall (x_k,y_k,z_k)\in\text{Spt}(\lambda),\ k=1,2.
        \end{align*}
        By cancelling out the terms $ \langle x_1,y_1\rangle$ and $ \langle x_2,y_2\rangle$ from both sides, we have
        \begin{align*}
            \langle x_1,z_1\rangle+\langle y_1,z_1\rangle+\langle x_2,z_2\rangle+\langle y_2,z_2\rangle\geq \langle x_1,z_2\rangle +\langle y_1,z_2\rangle+\langle x_2,y_2\rangle+\langle x_2,z_1\rangle+\langle y_2,z_1\rangle.
            \end{align*}   
By re-arranging the terms, we have
\begin{align*}
            \langle x_1,z_1\rangle-\langle x_1,z_2\rangle+\langle x_2,z_2\rangle-\langle x_2,z_1\rangle+\langle y_1,z_1\rangle-\langle y_1,z_2\rangle+\langle y_2,z_2\rangle-\langle y_2,z_1\rangle\geq 0,
            \end{align*}
from which the result follows.    
    \end{proof}
\end{lemma}

\noindent
\textbf{Proof of Theorem \ref{3MARG-QUADTHEOREM}.}
The proof follows from Theorems \ref{MAINTHEOREM} and \ref{MAINTHEOREMCP}. To this end, we shall show that an optimal plan $\lambda$ of \eqref{3MARG-QUAD}, and its restriction $\lambda_1$ to $X\times Y$, are $(c,P)$ and $(c_1,Q)$-extreme respectively where $P=\{\mathcal{S}_l\}_{l=1}^L$ and $Q:=\{Y_1,Y_2\}$ are considered as Borel ordered partitions of $Z$ and $Y$. Consider the following set-valued functions which are associated with the problem \eqref{3MARG-QUAD}
\begin{align*}
 & F_\lambda :X\times Y\to 2^Z,\quad  F_\lambda(x,y)=\{z\in Z\ :\ (x,y,z)\in\text{Spt}(\lambda)\},\\
 &   i:\text{Dom}(F_\lambda)\to \{1,\ldots,L\},\quad i(x,y)=\min\{j \ :\ F_{\lambda}(x,y)\cap \mathcal{S}_j\neq \emptyset\},\\
 & f_{\lambda,c,P}:\text{Dom}(F_\lambda)\to 2^{Z},\quad f_{\lambda,c,P}(x,y)=\text{argmax}\big\{c(x,y,z)\ :\ z\in F_\lambda(x,y)\cap\mathcal{S}_{i(x,y)}\big\}.
\end{align*}
 Recall that to show $\lambda$ is $(c,P)$-extreme we consider the function $c$ as a function on $(X\times Y)\times Z$. We now claim that
\begin{enumerate}
    \item\label{1-ext} For any $\lambda_1$-full measure subset $M_0\subseteq X\times Y$, we have
    \begin{align*}
         \text{Dom}(F_\lambda)\cap M_0=\text{Dom}(f_{\lambda,c,P})\cap M_0.
    \end{align*}

    \item\label{2-ext} For all different $(x,y),(x',y')\in\text{Dom}(F_\lambda)$,
\begin{align*}
    (F_\lambda(x,y)\setminus\{z\})\cap (F_\lambda(x',y')\setminus\{z'\})\cap Z_0=\emptyset,\quad \forall z\in f_{\lambda, c,P}(x,y),\ \forall z'\in f_{\lambda, c,P}(x',y'),
\end{align*}
where the $\gamma$-full measure subset $Z_0$ comes from Lemma \ref{ext-Gen}.
\end{enumerate}
Note that for each $(x,y)\in \text{Dom}(F_\lambda)$, the set $F_\lambda(x,y)$ is closed and for every $l\in\{1,\ldots,L\}$ the set $\mathcal{S}_l$ is compact. Thus, by continuity of $c$ we have 
\begin{align*}
    \text{argmax}\big\{c(x,y,z)\ :\ z\in F_\lambda(x,y)\cap\mathcal{S}_{i(x,y)}\big\}\neq \emptyset,
\end{align*}
which implies the case \eqref{1-ext}. To see \eqref{2-ext}, let it be otherwise and let
\begin{align*}
    z^*\in (F_\lambda(x,y)\setminus\{z\})\cap (F_\lambda(x',y')\setminus\{z'\})\cap Z_0.
\end{align*}
for some distinct $(x,y)$ and $(x',y')$ in $\text{Dom}(F_\lambda)$ and $z\in f_{\lambda,c,P}(x,y)$ and $z'\in f_{\lambda,c,P}(x',y')$. Therefore,
\begin{align*}
    (x,y,z), (x',y',z'), (x,y,z^*),(x',y',z^*)\in\text{Spt}(\lambda).
\end{align*}
Without loss of generality let $i(x,y)\leq i(x',y')$ and $z^*\in\mathcal{S}_j$, for some $j$. By the definition of the function $i$ we conclude that $j\geq i(x',y')$. Moreover, we have $z, z'\in \Omega_j\cup\mathcal{S}_j$, from which by strict convexity of $\Omega_j\cup\mathcal{S}_j$, we get
\begin{align}\label{strconv}
    &\langle \mathrm{n}(z^*), z^*-z\rangle > 0,\quad \text{and}\quad\langle \mathrm{n}(z^*), z^*-z'\rangle > 0.
\end{align}
By Lemma \ref{ext-Gen}, since we have $(x,y,z^*),(x',y',z^*)\in\text{Spt}(\lambda)$, there exists $\alpha\in\mathbb{R}$ so that 
\begin{align*}
\mathrm{n}(z^*)=\alpha(y'-y+x'-x).
\end{align*}
By substituting this into \eqref{strconv}, we obtain that
\begin{align*}
    &\alpha\langle y'-y+x'-x, z^*-z\rangle>0,\quad \text{and}\quad -\alpha\langle y-y'+x-x', z^*-z'\rangle>0.
\end{align*}
On the other hand, since
\begin{align*}
\{(x,y,z),(x',y',z^*)\}, \{(x',y',z'),(x,y,z^*)\}\subseteq\text{Spt}(\lambda),
\end{align*}
by Lemma \ref{LEMCYC} it can be concluded that
\begin{align*}
    &\langle y'-y+x'-x, z^*-z\rangle \geq 0,\quad \text{and}\quad \langle y-y'+x-x', z^*-z'\rangle \geq 0.
\end{align*}
This implies that $\alpha$ should be both positive and negative which is a contradiction. Hence, \eqref{2-ext} is satisfied and consequently, $\lambda$ is $(c,P)$-extreme.\\
Now, we do the same for $\lambda_1$. Recall the function $c_1$ from \eqref{c1-3MARGQUAD} and the reduced two-marginal optimization problem \eqref{c1-REDUC-PROB}. By Proposition \ref{PROP-3}, we already know that $\lambda_1$ is  an optimal plan of this problem. To apply Theorem \ref{MAINTHEOREMCP} we define the following set-valued functions to prove $(c_1,Q)$-extremality of $\lambda_1$,
\begin{align*}
    &F_{\lambda_1}:X\to 2^Y,\quad F_{\lambda_1}(x)=\{y\ : \ (x,y)\in\text{Spt}(\lambda_1)\},\\
    &i_1:\text{Dom}(F_{\lambda_1})\to\{1,2\},\quad i_1(x)=\min\left\{k\ :\ F_{\lambda_1}(x)\cap Y_k\neq\emptyset \right\},\\
    &f_{\lambda_1,c_1,Q}:X\to 2^Y,\quad f_{\lambda_1,c_1,Q}(x)=\argmax\{c_1(x,y)\ :\ y\in F_{\lambda_1}(x)\cap Y_{i_1(x)}\}.
\end{align*}
Set
\begin{align*}
    & X_0=\text{Dom}(D\phi_1),\quad  Y_0=Y_1\cup \text{Dom}(D\phi_2|_{Y_2}).
\end{align*}
Since $\mu$ and $\nu_2$ are absolutely continuous with respect to Lebesgue measure on $\mathbb{R}^{n-1}$ and the functions $\phi_1$ and $\phi_2$ are locally Lipschitz (on $\mathbb{R}^n$, and therefore on $\mathbb{R}^{n-1}$), hence, they are $\mu$- and $\nu_2$-a.e. differentiable, respectively. Therefore, 
\begin{align*}
    \mu(X_0)=\nu(Y_0)=1.
\end{align*}
To establish our desired result, we shall verify the two following conditions,
\begin{enumerate}
    \item We have
\begin{align*}
\text{Dom}(F_{\lambda_1})\cap X_0 =\text{Dom}(f_{\lambda_1,c_1,Q})\cap X_0.
\end{align*}
\item For any distinct $x,x'\in \text{Dom}(F_{\lambda_1})\cap X_0$, we have
\begin{align*}
    F_{\lambda_1}(x)\backslash\{y\}\cap F_{\lambda_1}(x')\backslash\{y'\}\cap Y_0=\emptyset,\quad \forall y\in f_{\lambda_1,c_1,Q}(x),\ \forall y'\in f_{\lambda_1,c_1,Q}(x').  
\end{align*} 
\end{enumerate}
The first condition holds true because $\mu$ and $\nu$ are compactly supported and $c_1$ is continuous.
Before going further, we note that for a fixed element $(x_0,y_0)\in \text{Spt}(\lambda_1)$, the function $x\mapsto c_1(x,y_0)$ is differentiable at $x_0\in \text{Dom}(F_{\lambda_1})\cap X_0$. This is due to the fact that the function $\tilde{c}_1(x,y)=\langle x,y\rangle$ is $\mu$-a.e. differentiable with respect to the first variable, and we have
\begin{align*}
    \psi(x+y)\leq \phi_1(x)+\phi_2(y)-\tilde{c}_1(x,y),
\end{align*}
and equality holds for $(x,y)\in\text{Spt}(\lambda_1)$, particularly for $(x_0,y_0)$. Then the differentiability of $x\mapsto \psi(x+y_0)$ at $x_0$ follows by Lemma \ref{LEMDIFF}, and hence, $D_{x}c_1$ exists at $(x_0,y_0)$. To prove the second condition let it be otherwise and 
\begin{align*}
  y^*\in  F_{\lambda_1}(x)\backslash\{y\}\cap F_{\lambda_1}(x')\backslash\{y'\}\cap Y_0,
\end{align*} 
for some distinct $x,x'\in\text{Dom}(F_{\lambda_1})\cap X_0$, and  $y\in f_{\lambda_1,c_1,Q}(x)$ and $y'\in f_{\lambda_1,c_1,Q}(x')$. Then we have
\begin{align*}
    (x,y),(x,y^*),(x',y'),(x',y^*)\in\text{Spt}(\lambda_1).
\end{align*}
We have two cases based on the position of $y^*$.
\begin{itemize}
    \item Case 1: $y^*\in Y_1$. In this case, by the definition of $f_{\lambda_1,c_1,Q}(x)$, we have that $y,y'\in Y_1$, as well. We note that since two planes $X$ and $Y_1$ are parallel, then we have
\begin{align}\label{ORTHO}
    \langle\mathrm{n}(x),y^*-y\rangle=0 .
\end{align}
On the other hand, since $(x,y),(x,y^*)\in\text{Spt}(\lambda_1)$, then by the existence of $D_xc_1(x,y)$ and $D_xc_1(x,y^*)$, we have
\begin{align}\label{DEFMONO}
    \mathrm{n}(x)=\alpha (y^*-y+D\psi(x+y^*)-D\psi(x+y)),
\end{align}
for some $\alpha\in\mathbb{R}$. Now substituting \eqref{DEFMONO} into \eqref{ORTHO} yields that $y=y^*$, which is a contradiction.

\item Case 2: $y^*\in Y_2$. In this case the derivative of $\phi_2(y^*)$ exists, from which we obtain
\begin{align}\label{DEFMONO22}
    \mathrm{n}(y^*)=\alpha (x-x'+D\psi(x+y^*)-D\psi(x'+y^*)),
\end{align}
for some $\alpha\in\mathbb{R}$. But similar to \eqref{ORTHO}, we have
\begin{align}\label{ORTHO22}
    \langle\mathrm{n}(y^*),x-x'\rangle=0.
\end{align}
\eqref{DEFMONO22} and \eqref{ORTHO22} imply that $x=x'$, which is a contradiction.
\end{itemize}
Hence, $\lambda_1$ is $(c_1,Q)$-extreme and by applying Theorem  \ref{MAINTHEOREMCP} the proof is completed.
\hfill \qedsymbol{}

\begin{remark}
\begin{enumerate}
    \item 
In Example \ref{EX-NON-UNIQ} we saw that if the three-marginal problem \eqref{EXEQ1} has a solution induced by a transport map then the solution is not unique. On the other hand, Theorem \ref{3MARG-QUADTHEOREM}  yields that \eqref{EXEQ1} has a unique solution. Therefore, it can be deduced that the problem \eqref{EXEQ1} has a unique solution which is not induced by a transport map. This actually distinguishes the Theorem \ref{3MARG-QUADTHEOREM} from Gangbo-\'Swi\c ech result. By the proof of Theorem \ref{3MARG-QUADTHEOREM} the measure $\lambda_1$, the restriction of the unique optimal plan $\lambda\in\Pi(\mu,\mu,\gamma)$ on $X\times Y$ is induced by a unique optimal map $T:X\to Y$ and the argument in Example \ref{EX-NON-UNIQ} shows that $T=\text{id}$. A careful examination reveals that the quadratic function in Example \ref{EX-NON-UNIQ} satisfies the $2$-twist condition. In fact by recalling Definition \ref{DEF-2-TWIST}, let $(x_0,x_0,z_0)\in\text{Spt}(\lambda)$ and $z\in L(x_0,x_0,z_0)$. Then there exists $\alpha\in\mathbb{R}$ such that
\begin{align*}
    z=z_0+\alpha\mathrm{n}(x_0).
\end{align*}
Hence, the element $z$ lives on the straight line passing through $z_0$ in the direction of $\mathrm{n}(x_0)$ which intersects $Z$ in at most two points. This shows that the function $c$ satisfies $2$-twist condition from which we conclude the existence of the following measurable maps
\begin{align*}
    \alpha :X\to [0,1],\quad G_1 :X\to Z,\quad\text{and}\quad G_2:X\to Z,
\end{align*}
such that
\begin{align*}
\lambda=(\text{id}\times \text{id}\times G_1)\#(\alpha\mu)+(\text{id}\times \text{id}\times G_2)\#((1-\alpha)\mu).
\end{align*}
\item It can be shown that the unique optimal plan $\lambda$ in Theorem \ref{3MARG-QUADTHEOREM} is concentrated on the union of the graph of $4L$ functions $T_j:X\to Y\times Z$, with $T_j=(S_{j},G_{j})$, where $S_{j}:X\to Y$ and $G_{j}:X\to Z$ as the cost function $c$ satisfies $4L$-twist condition. Hence, $\lambda$  is of the following form
\begin{align*}
    \lambda=\left(\sum_{j=1}^{4L}\alpha_j(x)\delta_{T_j(x)}\right)\otimes\mu=\left(\sum_{j=1}^{4L}\alpha_j(x)\delta_{S_{j}(x)}\times \delta_{G_{j}(x)}\right)\otimes\mu,
\end{align*}
where $\alpha_j:X\to [0,1]$ are measurable map as it is described in Definition \ref{DEF-MULTI-GRAPH} and Theorem \ref{MOMEN-CHAR-TH}.
\end{enumerate}
\end{remark}

\subsection{Combination of the Monge   and  the quadratic cost function}
For two Borel probability measures $\mu$ and $\nu$ on $X$ which is the closure of a bounded convex open subset of $\mathbb{R}^n$, the Monge problem which was initiated in \cite{MONGE}, is as follows
\begin{align}\label{MONGE-SECTION5}
    \min\bigg\{\int_X |x-T(x)|\,d\mu(x)\ :\  T\in \Lambda(\mu,\nu),\bigg\},
\end{align}
where $\Lambda(\mu,\nu)$ denotes the set of measurable transport maps that push forward $\mu$ to $\nu$. Due to non-linearity of the aforementioned target functional with respect to the variable $T$ and lack of proper compactness of the set $\Lambda(\mu,\nu)$ one may not expect the existence of a solution under mild assumption. By additional regularity hypothesis on $\mu$ and $\nu$, existence of a solution to \eqref{MONGE-SECTION5} is proven for the case of Lipschitz-continuous densities with respect to the Lebesgue measure $\mathcal{L}^n$ for $\mu$ and $\nu$ in \cite{EVANSGANGBO}, and integrable densities in \cite{AMBROSIO2} and \cite{TRUDINGERWANG}. A counter example in dimension $n=2$ is given in \cite{AMBROSIOCAFFARELLIBRENIERBUTTAZZOVILLANI} which shows that for any $s\in (1,2)$ there  exists a marginal $\mu$ which is absolutely continuous with respect to $\mathcal{H}^t$, the Hausdorff measure of dimension $t$ in $\mathbb{R}^2$, for any $t<s$ and is such that the Monge problem \eqref{MONGE-SECTION5} does not have a solution. 

In a more general setting, when $\mathbb{R}^n$ is equipped with an arbitrary norm $\|\cdot\|$, the problem \eqref{MONGE-SECTION5} has been studied for the cost function $\|x-T(x)\|$, by examining strict convexity and $C^2$-uniform convexity assumptions on the norm $\|\cdot\|$ or only under absolute continuity of the first marginal $\mu$ with respect to the Lebesgue measure on $\mathbb{R}^n$, see \cite{AMBROSIOKIRCHHEIMPRATELLI, CAFFARELLIFELDMANMCCANN, CARAVENNA, CHAMPIONDEPASCALE4, CHAMPIONDEPASCALE2, CHAMPIONDEPASCALE3, SUDAKOV}.

In this subsection, we consider the cost function which is constructed by adding the Monge cost to the quadratic one, 
\begin{align*}
    \quad c(x,y,z)=|x-y|+|x-z|^2+|y-z|^2.
\end{align*}
The associated three-marginal Monge-Kantorovich problem is given by
\begin{align}\label{MONG-QUAD}
        \inf_{\lambda\in\Pi(\mu,\nu,\gamma)}\int (|x-y|+|x-z|^2+|y-z|^2)\,d\lambda.
\end{align}
 Next theorem is our result regarding the 3-marginal problem \eqref{MONG-QUAD}.

\begin{theorem}
Let $X$, $Y$ and $Z$ be compact subsets of $\mathbb{R}^n$ which are supports of Borel probability measures $\mu$, $\nu$ and $\gamma$, respectively and $X\cap Y=\emptyset$. Assume that $\mu$ and $\gamma$ are absolutely continuous with respect to Lebesgue measure on $\mathbb{R}^n$. Then the problem \eqref{MONG-QUAD} admits a unique solution which is concentrated on the graph of a unique optimal transport map $T:X\to Y\times Z$.
\begin{proof}
To prove the claim, we  shall apply Theorem \ref{MAINTHEOREM}. Let 
\begin{align*}
    \phi_1\in L_1(X,\mu),\quad \phi_2\in L_1(Y,\nu),\quad\text{and}\quad\phi_3\in L_1(Z,\gamma).
\end{align*}
be the solution of the dual-Kantorovich problem to \eqref{MONG-QUAD}. Using the potential $\phi_3$, we define the function $c_1:X\times Y\to\mathbb{R}$ as follows
\begin{align*}
    &c_1(x,y)=|x-y|+\inf_{z\in Z}\big\{|x-z|^2+|y-z|^2-\phi_3(z)\big\}.
\end{align*}
We may rewrite $c_1$ as
\begin{align}\label{c_1MONGQUAD}
    &c_1(x,y)=\tilde{c}_1(x,y)-\psi(x+y),
\end{align}
where
\begin{align}\label{c_1MONGQUAD-SIDE}
    &\tilde{c}_1(x,y)=|x-y|+|x|^2+|y|^2,\quad \text{and}\quad \psi(x+y)=\sup_{z\in Z}\big\{2\langle x+y,z\rangle+|z|^2 -\phi_3(z)\big\}.
\end{align}
We shall now consider the reduced two-marginal optimization problem,
\begin{align}\label{C_1MONGQUAD-OPT}
    \inf_{\tau\in\Pi(\mu,\nu)}\int ( \tilde{c}_1(x,y)-\psi(x+y))\,d\tau .
\end{align}
We show that an optimal plan $\lambda$ of \eqref{MONG-QUAD} and its restriction $\lambda_1$ on $X\times Y$ are $c$ and $c_1$-extreme, respectively, where $c_1$ is given by \eqref{c_1MONGQUAD}. In fact, the sketch of the proof is as follows: first we show that there exists a unique map $T_0:X\times Y\to Z$ for which we have $\lambda=(\text{id}\times T_0)\#\lambda_1$ and then we prove that $\lambda_1=(\text{id}\times T_1)\#\mu$, for a unique optimal transport map $T_1:X\to Y$. At this aim, consider the two set-valued maps associated with $\lambda$ and $\lambda_1$,
\begin{align*}
    &F_{\lambda}:X\times Y\to 2^Z,\quad F_{\lambda}(x,y)=\{z\in Z\ : \ (x,y,z)\in\text{Spt}(\lambda)\},\\
     &  F_{\lambda_1}:X \to 2^Y,\quad F_{\lambda_1}(x)=\{y\in Y\ : \ (x,y)\in\text{Spt}(\lambda_1)\}.
\end{align*}
We note that since the marginals are compactly supported, then the two measures $\lambda$ and $\lambda_1$ have compact supports, as well. This implies that 
\begin{align}\label{SUPP-REL}
\pi_{XY}(\text{Spt}(\lambda))=\text{Spt}(\lambda_1), \quad\text{where}\quad\pi_{XY}:X\times Y\times Z\to X\times Y,
\end{align}
from which we have
\begin{align*}
    \pi_X(\text{Dom}(F_{\lambda}))=\text{Dom}(F_{\lambda_1}),\quad\text{where}\quad\pi_{X}:X\times Y\to X.
\end{align*}
Since $\mu$ is absolutely continuous with respect to Lebesgue measure, the potential $\phi_1$ is $\mu$-a.e. differentiable. Put $X_0=\text{Dom}(D\phi_1)$. We shall show that
\begin{enumerate}
    \item\label{SINGLAM}
for all $(x,y)\in \text{Dom}(F_\lambda)\cap(X_0\times Y)$, the set $F_\lambda(x,y)$ is a singleton,
    \item\label{SINGLAM1}
    for all $x\in \text{Dom}(F_{\lambda_1})\cap X_0 $, the set $F_{\lambda_1}(x)$ is singleton.
\end{enumerate}
To see \eqref{SINGLAM}, if  $z,z^*\in F_\lambda(x,y)$, then we have 
\begin{align*}
    D_xc(x,y,z)=D\phi_1(x)=D_xc(x,y,z^*),
\end{align*}
from which we get
\begin{align*}
    \frac{x-y}{|x-y|}+2x-2z=  \frac{x-y}{|x-y|}+2x-2z^*,
\end{align*}
and consequently, $z=z^*$. By having \eqref{SINGLAM}, it is obtained that there exists a unique measurable map $T_0:X\times Y\to Z$, such that $\lambda$ is supported on the graph of this map. \\
To prove \eqref{SINGLAM1}, first we observe that for $(x_0,y_0)\in\text{Spt}(\lambda_1)$ the function $x\mapsto \psi(x+y_0)$ is differentiable at $x_0$. Indeed, since the potentials $\phi_1$ and $\phi_2$ solves the dual problem to \eqref{C_1MONGQUAD-OPT}, we have
\begin{align*}
    c_1(x,y)\geq \phi_1(x)+\phi_2(y), \quad\text{with equality on}\; \text{Spt}(\lambda_1),
\end{align*}
or equivalently,
\begin{align*}
    \psi(x+y)\leq \tilde{c}_1(x,y)-\phi_1(x)-\phi_2(y), \quad\text{with equality on}\; \text{Spt}(\lambda_1).
\end{align*}
Now, by applying Lemma \ref{LEMDIFF} to the functions 
\begin{align*}
f(x)=\psi(x+y_0)\quad\text{and} \quad u(x)=\tilde{c}_1(x,y_0)-\phi_1(x)-\phi_2(y_0),
\end{align*}
due to $\mu$-almost everywhere differentiability of $\phi_1$ and existence of $D_x\tilde{c}_1(x_0,y_0)$ we conclude that $D\psi(x_0+y_0)$ exists. On top of that, since $c_1=\tilde{c}_1+\psi$, existence of $D_xc_1(x_0,y_0)$ is obtained. Therefore, 
\begin{align}\label{OO2}
    D_xc_1(x_0,y_0)=\frac{x_0-y_0}{|x_0-y_0|}+2x_0-D\psi(x_0+y_0).
\end{align}
Moreover, it can be easily shown that for every $(x_0,y_0)\in X\times Y$, we have
\begin{align}\label{OO3}
    \forall z_0\in\text{argmax}\big(2\langle x_0+y_0,z\rangle+|z|^2-\phi_3(z)\big)\quad\Longrightarrow\quad 2 z_0\in\partial \psi(x_0+y_0).
\end{align}
Let $x\in\text{Dom}(F_{\lambda_1})\cap X_0$. If there exist $y,y^*\in F_{\lambda_1}(x)$, then by the definition of $F_{\lambda_1}$, it means that
\begin{align}\label{OO1}
    (x,y),(x,y^*)\in\text{Spt}(\lambda_1)
\end{align}
Consequently, via \eqref{SUPP-REL}, there exist $z,z^*\in Z$ such that
\begin{align*}
    (x,y,z),(x,y^*,z^*)\in\text{Spt}(\lambda).
\end{align*}
Without loss of generality, here we can assume that $z,z^*\in Z_0=\text{Dom}(D\phi_3)$. This is due to the fact that through absolute continuity of $\gamma$, the potential $\phi_3$ is $\gamma$-almost everywhere differentiable which yields that $Z_0$ is a $\gamma$-full measure subset of $Z$.
From \eqref{OO1}, we get
\begin{align*}
    D_xc_1(x,y)=D\phi_1(x)=D_xc_1(x,y^*),
\end{align*}
and via \eqref{OO2} it is obtained that
\begin{align*}
    \frac{x-y}{|x-y|}-D\psi  (x+y)=\frac{x-y^*}{|x-y^*|}-D\psi  (x+y^*).
\end{align*}
By rearranging the above equality and multiplying by $y-y^*$, we have
\begin{align*}
0\geq\langle \frac{y^*-x}{|y^*-x|} -\frac{y-x}{|y-x|},y-y^*\rangle=\langle D\psi  (x+y)-D\psi  (x+y^*),y-y^*\rangle\geq 0.
\end{align*}
Hence,
\begin{align*}
    \langle D\psi  (x+y)-D\psi  (x+y^*),y-y^*\rangle= 0,
\end{align*}
and together with the fact that $D\psi  (x+y)=2z$ and $D\psi  (x+y^*)=2z^*$, we have
\begin{align}\label{EQQQ}
    \langle z-z^*,y-y^*\rangle= 0.
\end{align}
On the other hand, since the supremum in the definition of $\psi (x+y)$ happens at $z$ (similarly for $\psi (x+y^*)$ and $z^*$) we have that
\begin{align}
    &2\langle x+y,z\rangle+|z|^2 -\phi_3(z)\geq 2\langle x+y,z^*\rangle+|z^*|^2 -\phi_3(z^*),\label{OO4}\\
    &2\langle x+y^*,z^*\rangle+|z^*|^2 -\phi_3(z^*)\geq 2\langle x+y^*,z\rangle+|z|^2 -\phi_3(z).\label{OO5}
\end{align}
Adding these two inequalities yields that
\begin{align*}
     \langle z-z^*,y-y^*\rangle\geq 0.
\end{align*}
Since \eqref{EQQQ} holds, then we find out that the aforementioned inequalities \eqref{OO4} and \eqref{OO5} are indeed equalities. This yields that
\begin{align*}
    z^*\in\text{argmax}\big(2\langle x+y ,z\rangle+|z|^2-\phi_3(z)\big).
\end{align*}
Therefore, via \eqref{OO3}, it is obtained that
\begin{align*}
2z^*\in\partial\psi (x+y)=\{D\psi (x+y)\}=\{2z\},
\end{align*}
Hence, $z=z^*$. This implies that
\begin{align*}
    (x,y,z),\; (x,y^*,z)\in\text{Spt}(\lambda).
\end{align*}
Moreover, since $z\in Z_0$, then we have
\begin{align*}
    D_zc(x,y,z)=D\phi_3(z)=D_zc(x,y^*,z),
\end{align*}
from which we obtain that
\begin{align*}
    4z-2(x+y)=4z-2(x+y^*),
\end{align*}
and consequently, $y=y^*$. Thus, the set $F_{\lambda_1}(x)$ is a singleton and there exists a unique map $T_1:X\to Y$ for which we have $\lambda_1=(\text{id}\times T_1)\#\mu$. Now, if one defines
\begin{align*}
    T:X\to Y\times Z,\quad T(x)=\big(T_1(x),(T_0(x,T_1(x))\big),
\end{align*}
then it can be deduced that
\begin{align*}
    \lambda=(\text{id}\times T)\#\mu,
\end{align*}
which is the unique solution to the three-marginal Monge problem.

\end{proof}
\end{theorem}

\subsection{Gangbo-\'Swi\c ech  Theorem}
An immediate consequence of Theorem \ref{BIGTH} is the result obtained by Gangbo and \'Swi\c ech \cite{GANGBOSWIECH}.  Here, we state their result together with some other results related to the quadratic cost. Indeed, in this subsection we study the minimization problem associated with the cost function
\begin{align*}
c(x_1,\ldots,x_N)=\frac{1}{2}\sum_{i=1}^{N-1}\sum_{j=i+1}^N |x_i-x_j|^2,
\end{align*}
or equivalently, maximization problem corresponding to
\begin{align*}
c(x_1,\ldots,x_N)=\sum_{i=1}^{N-1}\sum_{j=i+1}^N \langle x_i,x_j\rangle,
\end{align*}
which are known in some terminology as \textit{attractive} and \textit{surplus} cost functions, respectively. In \cite{PASSVARGAS}, the authors investigated the surplus cost function via graph theory. They selected various choices of $\langle x_i,x_j\rangle$ which led to the existence of the solutions of Monge problem. Another cost function which has captivated researchers' attention is the \textit{repulsive} cost 
\begin{align*}
c(x_1,\ldots,x_N)=-\frac{1}{2}\sum_{i=1}^{N-1}\sum_{j=i+1}^N |x_i-x_j|^2.
\end{align*}
A characterization of the support of the minimizing optimal plans for multi-marginal Monge-Kantorovich problem associated with the repulsive cost is obtained in \cite{DIMARINOGELORINNENNA}. In fact, it is established that if a plan is concentrated on some hyperplane
\begin{align*}
\{(x_1,\ldots,x_N)\ :\ x_1+\cdots+x_N=k\},
\end{align*}
then it is a minimizing optimal plan. Particularly, in \cite{GELORINKAUSAMORAJALA}, for the case $N=3$, through defining specific marginals, it is shown that the minimization problem associated to the repulsive cost function admits a unique solution which is not concentrated on the graph of a function. However, for the case $N=2$, both attractive and repulsive cost functions admit a unique solution obtained by a transport map (see \cite{BRENIER,GELORINKAUSAMORAJALA}). Consider the following optimization problem
\begin{align}\label{ATT-QUAD}
    \inf_{\lambda\in\Pi(\mu_1,\ldots,\mu_N)}\int\frac{1}{2} \sum_{i=1}^{N-1}\sum_{j=i+1}^N |x_i-x_j|^2\,d\lambda.
\end{align}
Now, we state the result obtained in \cite{GANGBOSWIECH}, where for a convex function $f:\mathbb{R}^n\to\mathbb{R}$, the Legendre transform of $f$ is denoted by $f^*$.

\begin{theorem}[Gangbo-\'Swi\c ech]
For $k\in\mathcal{N}=\{1,\ldots,N\}$ let $\mu_k$ be Borel probability measure on $\mathbb{R}^n$ with finite second moment and vanishes on $(n-1)$-rectifiable set.  Then
\begin{enumerate}
    \item the dual problem to \eqref{ATT-QUAD} admits maximizers $u_k$ which are $\mu_k$-differentiable,
\item there exists a unique minimizer $\lambda$ to \eqref{ATT-QUAD} such that $\lambda=(\text{id}, T_2,\ldots, T_{N})\#\mu_1$,
and the maps $T_k: \mathbb{R}^n\to \mathbb{R}^n$ are defined by 
\begin{align*}
T_k(x_1) = Df^*_{k}\big(Df_1(x_1)\big),\quad k\in\mathcal{N}\backslash\{1\},
\end{align*}
where,
\begin{align*}
   &f_k(x_k)=\frac{1}{2}|x_k|^2+\phi_k(x_k),\quad\text{and}\quad \phi_k(x_k)=\frac{N-1}{2}|x_k|^2-u_k(x_k).
\end{align*}
\end{enumerate}
\end{theorem}
To present our result, we consider the equivalent problem to \eqref{ATT-QUAD}, i.e.
\begin{align}\label{SURP-QUAD}
    \sup_{\lambda\in\Pi(\mu_1,\ldots,\mu_N)}\int \sum_{i=1}^{N-1}\sum_{j=i+1}^N \langle x_i,x_j\rangle\, d\lambda .
\end{align}
To present our proof, we impose absolute continuity assumption on the first marginal $\mu_1$, instead of vanishing on $(n-1)$-rectifiable sets. However, with the latter assumption, our proof would work as well.
\begin{theorem}
    Let $\{(\mu_k)\}_{k\in\mathcal{N}}$ be family of Borel probability measures on $\mathbb{R}^n$  which have finite second moments and $\mu_1$ is absolutely continuous. Then the problem \eqref{SURP-QUAD} admits a unique solution which is induced by a transport map $T=(\text{id},T_2,\ldots,T_N)$ as follows
    \begin{align}\label{T-MAPS}
        T_j(x_1)=Du^*_j(x_1+D\phi_1(x_1))-x_1,
    \end{align}
    where $ (\phi_1,\ldots,\phi_N)$ are solutions to the dual of  \eqref{SURP-QUAD}, and $u_j$'s are strictly convex functions.
    \begin{proof}
    Since the surplus cost function is continuous and the measures $\mu_k$ have finite second moments, then dual problem to \eqref{SURP-QUAD} admits a solution
    \begin{align*}
        (\phi_1,\ldots,\phi_N)\in L_1(\mathbb{R}^n,\mu_1)\times\cdots\times L_1(\mathbb{R}^n,\mu_N).
    \end{align*}
Put $X_k=\text{Spt}(\mu_k)$ and consider the following two marginals optimization problems which are obtained from \eqref{SURP-QUAD} by reduction method in which we have chosen $\mathcal{P}_j=\{1,j\}$,
\begin{align}\label{GANGBO-SWIECH-RED}
    \sup_{\tau_{1j}\in\Pi(\mu_1,\mu_j)}\int c_{1j}(x_1,x_j) d\,\tau_{1j}(x_1, x_j),\quad\text{with}\quad c_{1j}=\langle x_1,x_j\rangle+\psi_{j}(x_1+x_j),
\end{align}
where
\begin{align*}
    \psi_j(x_1+x_j)=\sup_{\prod_{k=1,\neq j}^N X_k}\bigg\{\langle x_1+x_j,\sum_{k=2,k\neq j}^{N}x_k\rangle+\sum_{i=2,\neq j}^{N-1}\sum_{k=i+1,\neq j}^{N}\langle x_i,x_k\rangle-\sum_{k=2,\neq j}^{N}\phi_k(x_k)\bigg\},\quad j\in\mathcal{N}\backslash\{1\}.
\end{align*}

First, we shall show that $\phi_1$ is $\mu_1$-almost everywhere differentiable, then we conclude differentiability of the map $x_1\mapsto\Psi_j(x_1+x_j)$. We note that $\phi_j$ is lower semi-continuous function and can be  obtained from the following formula
\begin{align*}
    \phi_j(x_j)=\sup_{\prod_{k=1,\neq j}^N X_k}\bigg\{\sum_{i=1}^{N-1}\sum_{j=i+1}^N \langle x_i,x_j\rangle-\sum_{k=1,\neq j}^N\phi_k(x_k)\bigg\}.
\end{align*}
Hence, it is a convex function. Moreover, since $\phi_j\in L_1(\mathbb{R}^n,\mu_j)$ so it is finite $\mu_j$-almost everywhere, and by convexity of $\phi_j$ it is continuous in the interior of its effective domain. Consequently, due to absolute continuity of $\mu_1$, the function $\phi_1$ is $\mu_1$-almost everywhere differentiable. Put $X_{01}=\text{Dom}(D\phi_1)$ which is a $\mu_1$-full measure subset of $X_1$. Let $\lambda$ be an optimal plan of \eqref{SURP-QUAD} and $\pi_{1j}$ denote the projection map from $\prod_{k=1}^N X_k$ onto $X_1\times X_j$. By Proposition \ref{PROP-3}, the measures $\lambda_{1j}=\pi_{1j}\#\lambda$ are optimal plans for \eqref{GANGBO-SWIECH-RED}, for $j\in\mathcal{N}\backslash\{1\}$. We observe that for a fixed $j$, through Lemma \ref{LEMDIFF}, the derivatives $D_{x_1}c_{1j}(x_1,x_j)$ and $D\psi_j(x_1+x_j)$ exist for $(x_1,x_j)\in \text{Spt}(\lambda_{1j})$. Consider the following set-valued map
        \begin{align*}
            F_{\lambda_{1j}}:X_1\to 2^{X_j},\quad F_{\lambda_{1j}}(x_1)=\{x_j\ :\ (x_1,x_j)\in\text{Spt}(\lambda_{1j})\}.
        \end{align*}
        We shall show that for all $x_1\in\text{Dom}(F_{\lambda_{1j}})\cap X_{01}$, the set $F_{\lambda_{1j}}(x_1)$ is singleton. In fact, if $x_j,x_j^*\in F_{\lambda_{1j}}(x_1)$, then we have
        \begin{align*}
            D_{x_1}c_{1j}(x_1,x_j)= D_{x_1}c_{1j}(x_1,x^*_j),
        \end{align*}
and consequently,
\begin{align*}
    x_j-x^*_j=D\psi_j(x_1+x^*_j)-D\psi_j(x_1+x_j).
\end{align*}
        Multiplying last relation by $x_j-x^*_j$ we have
\begin{align*}
    0\leq \|x_j-x^*_j\|=\langle D\psi_j(x_1+x^*_j)-D\psi_j(x_1+x_j), x_j-x^*_j\rangle\leq 0.
\end{align*}
Hence, $x_j=x^*_j$. This means that there exists a unique measurable map $T_j:X_1\to X_j$ such that 
\begin{align*}
\lambda_{1j}=(\text{id}\times T_j)\#\mu_1,
\end{align*}
is the unique solution to \eqref{GANGBO-SWIECH-RED}, for $j\in\mathcal{N}\backslash\{1\}$. Now, the result follows via Theorem \ref{BIGTH} from which we have
\begin{align*}
    \lambda=(\text{id}\times T_2\times\cdots\times T_N)\#\mu_1.
\end{align*}
To determine the explicit expression for the maps $T_j$, due to the fact that 
\begin{align*}
\langle x_1,x_j\rangle+\psi_j(x_1+x_j)=\phi_1(x_1)+\phi(x_j),\quad\text{on}\ \text{Spt}(\lambda_{1j}),
\end{align*}
and differentiability with respect to the first variable, we have that
\begin{align*}
    x_j+D\psi_j(x_1+x_j)=D\phi_1(x_1).
\end{align*}
Consequently,
\begin{align*}
   x_1+ x_j+D\psi_j(x_1+x_j)=x_1+D\phi_1(x_1).
\end{align*}
On the other hand, since the function $u_j(t)=\frac{1}{2}|t|^2+\psi_j(t)$ is strictly convex, then its conjugate function $u^*_j$ is differentiable. From this, we get
\begin{align*}
    Du_j(x_1+x_j)=x_1+D\phi_1(x_1),
\end{align*}
and therefore,
\begin{align*}
    x_j=Du^*_j(x_1+D\phi_1(x_1))-x_1.
\end{align*}
Therefore, the optimal transport map $T_j$ is determined as \eqref{T-MAPS}.

    \end{proof}

\end{theorem}

\section{Appendix}\label{SectionAppendix}
This section is devoted to the prerequisite of the disintegration theorems and some side results which are used in this paper. We begin with the disintegration notions. At this aim, we use \cite{AMBROSIO,DELLACHERIEMEYER} as the main references. As a direct summary one can also refer to \cite{AMBROSIOCAFFARELLIBRENIERBUTTAZZOVILLANI}. It should be noted that we may state definitions and results in special cases. More general case can be found in the given references.

For a space $X$, the set of Borel probability measures on $X$ is denoted by $P(X)$. Recall the projection map $\pi_{XY}:X\times Y\times Z\to X\times Y$, for a measure $\lambda\in\Pi(\mu,\nu,\gamma)$, we denote by $\lambda^{XY}$ the measure $\pi_{XY}\#\lambda$, i.e. the restriction of $\lambda$ on $X\times Y$. We do the same in the case of $X\times Z$, as well.

The following results are side works in our paper and they will be partially applied  to parts of our proofs. For a fixed measure $\lambda\in\Pi(\mu,\nu,\gamma)$, by considering its marginal $\lambda^{XY}$ on $X\times Y$, by $\Pi(\lambda^{XY},\gamma)$, we mean the convex subset of $\Pi(\mu,\nu,\gamma)$ of measures with double restriction on $X\times Y$ equal to $\lambda^{XY}$. Moreover, $\Pi(\lambda^{XY},\gamma)\cap \Pi(\lambda^{XZ},\nu)$ is the convex subset of $\Pi(\mu,\nu,\gamma)$ containing measures with restrictions on $X\times Y$ and $X\times Z$ equal to $\lambda^{XY}$ and $\lambda^{XZ}$, respectively. Recall the description given in the beginning of the Section \ref{SectionDisintegration}, for a measure $\lambda\in\Pi(\mu,\nu,\gamma)$ we have
\begin{align*}
    \lambda=\lambda^x\otimes\mu,\quad \lambda^x\in\Pi(\nu^x,\gamma^x)\subseteq P(Y\times Z),
\end{align*}
where
\begin{align*}
    \lambda^{XY}=\nu^x\otimes\mu,\quad\nu^x\in P(Y),\quad\text{and}\quad \lambda^{XZ}=\gamma^x\otimes\mu,\quad\gamma^x\in P(Z).
\end{align*}

\begin{lemma}\label{LEM2}
An element $\lambda=\lambda^x\otimes\mu\in \Pi(\mu,\nu,\gamma)$ is an extreme point of $\Pi(\lambda^{XY},\gamma)\cap \Pi(\lambda^{XZ},\nu)$ whenever $\lambda^x$ is an extreme point of $\Pi(\nu^x,\gamma^x)$, for $\mu$-almost every $x\in X$.
\begin{proof}
If there exist $\lambda_i=\lambda_i^x\otimes\mu\in\Pi(\lambda^{XY},\gamma)\cap \Pi(\lambda^{XZ},\nu)$, such that 
\begin{align*}
    \lambda=\frac{1}{2}(\lambda_1+\lambda_2),\quad \lambda\neq\lambda_1\neq\lambda_2,
\end{align*}
then by the uniqueness of disintegration we have
\begin{align}\label{EQU-ACC}
    \lambda^x=\frac{1}{2}(\lambda^x_1+\lambda^x_2),\quad  \lambda^x_i\in \Pi(\nu^x,\gamma^x),\quad \mu\text{-a.e.}\; x\in X.
\end{align}
Now, since by assumption for $\mu$-almost every $x\in X$ the measure $\lambda^x$ is an extreme point of $\Pi(\nu^x,\gamma^x)$ then we have
\begin{align*}
    \lambda^x=\lambda^x_1=\lambda^x_2,\quad \mu\text{-a.e.}\; x\in X,
\end{align*}
which implies that $\lambda=\lambda_1=\lambda_2$.
\end{proof}
\end{lemma}

\begin{proposition}\label{PROP-A}
Let $\lambda\in \Pi(\mu,\nu,\gamma)$.
\begin{enumerate}
\item\label{P1}
If $\lambda$ is an extreme point of $\Pi(\mu, \nu,\gamma)$ then it is an extreme point of $\Pi(\lambda^{XY},\gamma)$, $\Pi(\lambda^{YZ},\mu)$, and $\Pi(\lambda^{XZ},\nu)$
\item\label{P2}
If $\lambda$ is an extreme point of $\Pi(\lambda^{XY},\gamma)$ and $\lambda^{XY}$ is an extreme point of $\Pi(\mu,\nu)$, then it is an extreme point of $\Pi(\mu, \nu,\gamma)$.
    \item\label{P3}
    Assume that $\lambda$ admits a disintegration of the form 
    \begin{align}\label{DISINT}
    \lambda=(\gamma^x\times \nu^x)\otimes\mu,\quad \text{where}\quad\lambda^{XY}=\nu^x\otimes \mu \ \text{and}\ \lambda^{XZ}=\gamma^x\otimes \mu.
\end{align}
\begin{enumerate}
    \item\label{PA}
            If $\lambda$ is an extreme point of $\Pi(\lambda^{XY},\gamma)$ then $\lambda^{XZ}$ is an extreme point of $\Pi(\mu,\gamma)$
            \item\label{PB}
             If $\lambda$ is an extreme point of $\Pi(\lambda^{XY},\gamma)$ and $\Pi(\lambda^{XZ},\mu)$ then $\lambda$ is an extreme point of $\Pi(\mu,\nu,\gamma)$.
    
\end{enumerate}
\end{enumerate}

\begin{proof}
\begin{enumerate}
    \item
    This is obvious, since each of $\Pi(\lambda^{XY},\gamma)$, $\Pi(\lambda^{YZ},\mu)$, and $\Pi(\lambda^{XZ},\nu)$ is a subset of $\Pi(\mu,\nu,\gamma)$.
\item 
If $\lambda=\frac{1}{2}\lambda_1+\frac{1}{2}\lambda_2$ for $\lambda_i\in\Pi(\mu,\nu,\gamma)$, then we have
\begin{align*}
   \lambda^{XY}=\frac{1}{2}\lambda^{XY}_1+\frac{1}{2}\lambda^{XY}_2.
\end{align*}
But, the measure $\lambda^{XY}$ is an extreme point of $\Pi(\mu,\nu)$. Therefore, $   \lambda=\lambda^{XY}_1=\lambda^{XY}_2$. This implies that
\begin{align*}
    \lambda_i\in\Pi(\lambda^{XY},\gamma).
\end{align*}
Now, due to extremality of $\lambda$ in $\Pi(\lambda^{XY},\gamma)$ we obtain a contradiction.
\item We just need to show Part \eqref{PA} then Part \eqref{PB} will be deduced from Part \eqref{P2}.
\begin{enumerate}
    \item 
If  $\lambda^{XZ}=\gamma^x\otimes\mu$ is not an extreme point of $\Pi(\mu,\gamma)$ then 
\begin{align*}
    &\lambda^{YZ}=\frac{1}{2}\theta_1+\frac{1}{2}\theta_2,\quad \text{for}\ \theta_i=\gamma_i^x\otimes\mu\in\Pi(\mu,\gamma).
\end{align*}
and consequently,
\begin{align*}
    \gamma^y\otimes\nu&=(\frac{1}{2}\gamma^x_1+\frac{1}{2}\gamma^x_2)\otimes\mu,\quad \mu\text{-a.e.}\; x.
\end{align*}
By uniqueness of disintegration, we have
\begin{align*}
    \gamma^x=\frac{1}{2}\gamma^x_1+\frac{1}{2}\gamma^x_2,\quad \mu\text{-a.e.}\; x.
\end{align*}
Define
\begin{align*}
    \Tilde{\lambda}_i=(\nu^x\times\gamma^x_i)\otimes \mu,\quad i=1,2.
\end{align*}
Then $\Tilde{\lambda}_i\in \Pi(\lambda^{XY},\gamma)$ for which we have
\begin{align*}
\Tilde{\lambda}^{XZ}_i=\theta_i,\quad\text{and}\quad    \lambda=\frac{1}{2}\Tilde{\lambda}_1+\frac{1}{2}\Tilde{\lambda}_2.
\end{align*}
By extremality of $\lambda$ in $\Pi(\lambda^{XY},\gamma)$, we then deduce that $\lambda=\Tilde{\lambda}_i$, and
\begin{align}
    \lambda^{XZ}=\theta_i.
\end{align}
\item 
By assumption, $\lambda$ is an extreme point of $\Pi(\lambda^{XY},\gamma)$ so by Part \eqref{PA} its marginal on $X\times Z$ is an extreme point of $\Pi(\mu,\gamma)$. On the other hand, $\lambda$ is an extreme point of $\Pi(\lambda^{XZ},\nu)$, therefore, by Part \eqref{P2} we have the result.
\end{enumerate}

\end{enumerate}
\end{proof}

\end{proposition}

The last result which is an applicable lemma has been used frequently in the proofs of Section \ref{SectionApplication}.
\begin{lemma}\label{LEMDIFF}
Assume that $f:\mathbb{R}^n\to\mathbb{R}\cup\{+\infty\}$ is a convex function, and  $x_0\in\mathbb{R}^n$. Let $u:\mathbb{R}^n\to\mathbb{R}$ be a function such that $f(x_0)=u(x_0)$ and on a neighborhood $N_r(x_0)$ of $x_0$ we have
\begin{align*}
    f(x)\leq u(x),\quad x\in N_r(x_0).
\end{align*}
If $u$ is differentiable at $x_0$ then the function $f$ is differentiable at $x_0$ as well.
\end{lemma}
\begin{proof}
Since $u$ is differentiable at $x_0$ then it is finite, so via $f(x_0)=u(x_0)$ the function $f$ is finite at $x_0$ as well. Moreover, $f$ is bounded above in the neighborhood $N_r(x_0)$. So, from \cite{EKELANDTEMAM}, it can be concluded that $f$ is a proper and locally Lipschitz function with $\partial f(x_0)\neq\emptyset$. Now, let $p\in\partial f(x_0)$. We shall show that $p=Du(x_0)$. To this end, we first note that for all $h\in\mathbb{R}^n$, we have
    \begin{align*}
        u(x_0+h)-u(x_0)\geq f(x_0+h)-f(x_0)\geq \langle p, h\rangle.
    \end{align*}
    Subtracting $\langle Du(x_0),h\rangle$ from all sides and dividing by $|h|$, it is obtained that
    \begin{align*}
        \frac{u(x_0+h)-u(x_0)-\langle Du(x_0),h\rangle}{|h|}\geq \frac{\langle p-Du(x_0),h\rangle}{|h|}.
    \end{align*}
    Now, if we put $h=t(p-Du(x_0))$ then as $t\to 0^+$, we have
    \begin{align*}
        0\geq |p-Du(x_0)|.
    \end{align*}
    Hence, $p=Du(x_0)$. This yields that $\partial f(x_0)$ is singleton, and therefore, $f$ is differentiable at $x_0$.
\end{proof}

\section*{Declaration of interest}

The authors declare that they have no known competing financial interests or personal relationships that could have
appeared to influence the work reported in this paper.


\section*{Data availability statement}
Data sharing not applicable to this article as no datasets were generated or analysed during the current study.

\end{document}